\documentclass[a4paper,12pt,reqno]{amsart}

\pdfoutput=1

\usepackage[margin=2.5cm, headsep=1cm, footskip=1cm]{geometry}
\usepackage[hang, flushmargin]{footmisc}
\usepackage[english]{isodate}
\usepackage{amsmath, amsthm, amssymb, mathtools}
\usepackage{enumitem}
\usepackage{xcolor}
\usepackage{tikz, tikzsymbols, tikz-cd}
\usepackage{multicol}
\usepackage[colorlinks=true, urlcolor=black, citecolor=black, linkcolor=black, hyperfootnotes=true]{hyperref}
\usepackage[capitalise, nameinlink, noabbrev, nosort]{cleveref}

\definecolor{indigo}{HTML}{492DA5}

\allowdisplaybreaks[2]
\setenumerate{listparindent=\parindent}

\providecommand{\noopsort}[1]{}

\makeatletter\g@addto@macro\bfseries{\boldmath}\makeatother

\makeatletter
\let\origsection\section
\renewcommand\section{\@ifstar{\starsection}{\nostarsection}}
\newcommand\sectionspace{\vspace{0.5ex}}
\newcommand\nostarsection[1]{\sectionspace\origsection{#1}\sectionspace}
\newcommand\starsection[1]{\sectionspace\origsection*{#1}\sectionspace}
\makeatother

\crefname{page}{page}{pages}

\setlist[enumerate]{font=\normalfont}
\crefname{enumi}{}{}
\crefname{enumii}{}{}

\numberwithin{equation}{section}
\crefname{equation}{equation}{equations}

\crefname{condition}{condition}{conditions}
\creflabelformat{condition}{#2(#1)#3}

\newtheorem{theorem}{Theorem}[section]

\newtheorem{thm}[theorem]{Theorem}
\crefname{thm}{Theorem}{Theorems}

\newtheorem{lemma}[theorem]{Lemma}
\crefname{lemma}{Lemma}{Lemmas}

\newtheorem{prop}[theorem]{Proposition}
\crefname{prop}{Proposition}{Propositions}

\newtheorem{cor}[theorem]{Corollary}
\crefname{cor}{Corollary}{Corollaries}

\theoremstyle{definition}

\newtheorem{definition}[theorem]{Definition}
\crefname{definition}{Definition}{Definitions}

\newtheorem{remark}[theorem]{Remark}
\crefname{remark}{Remark}{Remarks}

\crefname{remarks}{Remarks}{Remarks}

\newtheorem{example}[theorem]{Example}
\crefname{example}{Example}{Examples}

\newcommand{\T}{\mathbb{T}}

\newcommand{\VV}{\mathcal{U}}

\newcommand{\Go}{{G^{(0)}}}

\newcommand{\supp}{\operatorname{supp}}

\newcommand{\lal}{\langle}
\newcommand{\ral}{\rangle}

\newcommand{\Bc}{B_c}
\newcommand{\Nc}{N_c}

\begin{document}

\date{\today}
\title{Cartan semigroups and twisted groupoid C*-algebras}
\author[Bice]{Tristan Bice}
\author[Clark]{Lisa Orloff Clark}
\author[Lin]{Ying-Fen Lin}
\author[McCormick]{Kathryn McCormick}

\thanks{Research of the first-named author was supported by GAČR project 22-07833K and RVO: 67985840. The second-named author was supported by Marsden grant  21-VUW-156 from the Royal Society of New Zealand. The third-named author thanks Victoria University of Wellington for their hospitality during her visit, funded by Queen's University Belfast, to the second-named author. The fourth-named author also thanks the second-named author and Victoria University of Wellington for their hospitality during her visit, as well as supplementing funding by CSULB.\\ 
The authors thank the anonymous referee for their thorough review of the paper and helpful suggestions.  }

\address[T. Bice]{Institute of Mathematics of the Czech Academy of Sciences, \v{Z}itn\'a 25, Prague, CZECH REPUBLIC}
\email{bice@math.cas.cz}

\address[L.O. Clark]{School of Mathematics and Statistics, Victoria University of Wellington, PO Box 600, Wellington 6140, NEW ZEALAND}
\email{\href{mailto:lisa.orloffclark@vuw.ac.nz}{lisa.orloffclark@vuw.ac.nz}}

\address[Y.-F. Lin]{Mathematical Sciences Research Centre, Queen's University Belfast, Belfast, BT7 1NN, UNITED KINGDOM}
\email{\href{mailto:y.lin@qub.ac.uk}{y.lin@qub.ac.uk}}

\address[K. McCormick]{Department of Mathematics and Statistics, California State University, Long Beach, CA, UNITED STATES}
\email{\href{mailto:kathryn.mccormick@csulb.edu}{kathryn.mccormick@csulb.edu}}

\subjclass[2020]{46L05 (primary),
22A22 (secondary)}
\keywords{Twisted groupoid C*-algebra, Cartan subalgebra}

\begin{abstract}
    We prove that twisted groupoid C*-algebras are characterised, up to isomorphism, by having \emph{Cartan semigroups}, a natural generalisation of normaliser semigroups of Cartan subalgebras.  This extends the classic Kumjian-Renault theory to general twisted \'etale groupoid C*-algebras, even non-reduced C*-algebras of non-effective groupoids.
\end{abstract}

\maketitle

\setcounter{tocdepth}{1}
\tableofcontents

\section{Introduction}

Groupoid C*-algebras have been playing an ever increasing role in C*-algebra theory since their inception in \cite{Renault1980}.  Indeed, it is rare to meet a C*-algebra that can not be built in a natural way from some groupoid.  This begs the question -- what exactly distinguishes groupoid C*-algebras from arbitrary C*-algebras?  For reduced C*-algebras of twisted effective \'{e}tale groupoids, the Kumjian-Renault theory developed in \cite{Kumjian1986} and \cite{Renault2008} provides a complete answer -- these are (up to isomorphism) precisely the C*-algebras $A$ which have a \emph{Cartan subalgebra} $B$, i.e.~a maximal commutative C*-subalgebra (MASA) whose normalisers $N(B)$ generate $A$ and which is the range of a faithful conditional expectation $E:A\twoheadrightarrow B$ (in which case we call $(A,B)$ a \emph{Cartan pair}).  While a particular C*-algebra may arise from different twisted effective \'{e}tale groupoids, the choice of Cartan subalgebra $B$ completely determines the groupoid, which can be constructed directly from its normaliser semigroup $N(B)$.  Indeed, in the resulting twisted groupoid C*-algebra, $B$ gets identified with the \emph{diagonal subalgebra} consisting of functions supported on the unit space, while its normalisers $N(B)$ get identified with the \emph{monomial semigroup} consisting of functions supported on bisections.  Cartan subalgebras and their normalisers thus again completely characterise diagonal subalgebras and monomial semigroups of reduced C*-algebras of twisted effective \'{e}tale groupoids.

But what of non-effective \'etale groupoids?  Even recovering such a groupoid from its reduced C*-algebra and diagonal subalgebra then becomes an impossible task in general (e.g.~$\mathbb{Z}_4$ and $\mathbb{Z}_2\times\mathbb{Z}_2$ give rise to the same C*-algebra and diagonal, as noted in \cite{CRST2021}), although in some cases this can be achieved under some strong conditions on the isotropy, as in \cite{CRST2021}, or in the presence of a dual group action, as in \cite{BFPR}. In the algebraic setting of Steinberg algebras of ample groupoids, an analog of the Kumjian-Renault theory has been established in work of the last three authors and coauthors in \cite{ACCCLMRSS} which does indeed apply to some non-effective groupoids, namely those satisfying the `local bisection hypothesis'.  However, the C*-algebraic version of the local bisection hypothesis has been shown recently in work of the last three authors and coauthors in \cite{ABCCLMR2023} to only apply to effective groupoids, thus ruling out a similar approach in the C*-algebra setting.

Our goal here is to show that an entirely satisfactory extension of the Kumjian-Renault theory to C*-algebras of general twisted \'etale groupoids can be achieved by shifting focus from the diagonal subalgebra to the monomial semigroup.  This approach is in line with previous work of the first two authors in \cite{BC2021} showing how to recover general \'etale groupoids from `bumpy semigroups' of bisection-supported functions.  Similarly, the algebraic analog of the Kumjian-Renault theory mentioned above in \cite{ACCCLMRSS} has recently been extended by the first author to bundles over general ample groupoids in \cite{Bice2023} and their resulting Steinberg rings, which have a distinguished semigroup as a defining part of their structure.

Accordingly, here we introduce \emph{Cartan semigroups} as the appropriate generalisation of normaliser semigroups of Cartan subalgebras. Indeed, a Cartan semigroup $N$ is still required to generate the ambient C*-algebra $A$ and to contain a commutative \emph{semi-Cartan subalgebra} $B$ (generated by the positive elements of $N$), which is also the range of a conditional expectation $E:A\twoheadrightarrow B$.  However, we do not require $B$ to be a MASA -- instead we place a weaker `stability' condition on our expectation $E$ which still ensures its uniqueness and that the groupoid we build is Hausdorff.  We do not even need to assume our expectation $E$ is faithful and, as a result, our work applies not only to reduced C*-algebras but also full C*-algebras and all exotic C*-completions in between.

With these Cartan semigroups we are able to prove exactly the same kind of results even for non-effective groupoids.  Specifically, we show that twisted groupoid C*-algebras are completely characterised (up to isomorphism) by having a Cartan semigroup.  We show that the resulting semi-Cartan subalgebras $B$ characterise diagonal subalgebras, while \emph{summable} Cartan semigroups (those that are also closed under \emph{compatible sums}) characterise monomial semigroups on the groupoid side. Most of the paper is devoted to proving the following:

\begin{thm}\label{thm:main}[\cref{cor:it}]\label{MainResult}
Let $A$ be a C*-algebra containing a Cartan semigroup $N$ with semi-Cartan subalgebra $B$ generated by the positive elements of $N$ and a stable expectation $E:A\twoheadrightarrow B$ (see \cref{def:semiCartan}).  We then have an isomorphism $\Psi$ from $A$ onto a twisted groupoid C*-algebra $C=\mathrm{cl}(C_c(\Sigma;G))$.  Moreover,
    \begin{enumerate}
        \item If $N$ is summable then $\Psi(N)$ is the monomial semigroup $\mathrm{cl}(\Nc(\Sigma;G))$, where $N_c(\Sigma; G)= \{a\in C(\Sigma;G): q(\overline{\supp}(a)) \text{ is a compact bisection}\}$.
        \item\label{it2:it'} If $E$ is faithful then $C=\Psi(A)$ is the reduced C*-algebra $C_r^*(\Sigma;G)$.
    \end{enumerate}
\end{thm}

What is more, if we start with a twisted groupoid C*-algebra, then its monomial semigroup is a Cartan semigroup and its usual diagonal is a semi-Cartan subalgebra by \cref{TwistedGroupoid->CartanSemigroup}, and the original twist is recovered by \cref{thm:main}, see \cref{rmk:recover}. 

Incidentally, while the Cartan semigroup $N$ may be distinct from the normaliser semigroup $N(B)$ used in the Kumjian-Renault theory, $N$ is always contained in $N(B)$, by \cref{lem:Normal}.  In fact, if $N=N(B)$ then $B$ is a MASA, by \cref{NormalisersImpliesMASA}.  If the expectation $E$ is also faithful then, by \eqref{it2:it'} above, our result reduces to the classical Kumjian-Renault result.  Even in this case, however, our work has some added value, as we build our groupoids $\Sigma$ and $G$ using ultrafilters, in contrast to the original construction via a groupoid of germs. The advantage of ultrafilters is that they have a nice general theory which parallels the classical theory of maximal ideals.  Thus, even for the original normaliser semigroups of Cartan subalgebras, our work provides an alternative approach to the Kumjian-Renault theory.

The paper is organised as follows. In \cref{sec:Preliminaries}, we establish notation and preliminaries for twisted groupoid C*-algebras. In \cref{sec:Cartansemi}, we define Cartan semigroups and semi-Cartan subalgebras, and develop the building blocks we need. In \cref{sec:domination}- \cref{sectionG}, we introduce the domination relation and the groupoid of ultrafilters defined by this relation. Section \ref{sec:equivalence} lays the groundwork for the twist, while Section \ref{sec:sigma} defines the twist and quotient map. In \cref{sec:rep}, we conclude our goal of representing a semi-Cartan pair as a twisted groupoid C*-algebra, and Section \ref{sec:masa} spells out the specific case when we have a MASA. In the final section, we compare some variants of the domination relation.

\section{Preliminaries} \label{sec:Preliminaries}

\subsection{Normed Spaces}\label{NormedSpaces}

As usual, a \emph{norm} on a complex vector space $A$ is a subadditive function $\|\cdot\|:A\rightarrow\mathbb{R}_+$ which is non-zero on $A\setminus\{0\}$ and satisfies $\|za\|=|z|\|a\|$, for all $z\in\mathbb{C}$ and $a\in A$.  The unit ball of $A$ with respect to a given norm $\|\cdot\|$ is then denoted by
\[A^1=\{a\in A:\|a\|\leq1\}.\]
A \emph{contraction} is a linear map $D:A\rightarrow B$ between normed spaces with $D(A^1)\subseteq B^1$.  If an operator $D: A\rightarrow A$ is contractive with respect to some norm $\|\cdot\|$ then we call that norm \emph{$D$-contractive}.  If $D:A\rightarrow A$ is contractive and idempotent (that is, $D \circ D = D$), then we call it an \emph{expectation}.  If $A$ is a *-algebra then a \emph{C*-norm} on $A$ is a submultiplicative norm with
\[\|a^*a\|=\|a\|^2,\]
for all $a\in A$.  The completion of $A$ with respect to any C*-norm is then a C*-algebra. Note that by {\cite[II.6.10.3]{Blackadar2017}} an expectation between C*-algebras is a conditional expectation in the traditional sense of \cite[II.6.10.1]{Blackadar2017}.

For any function $f:X\rightarrow\mathbb{C}$, we denote its \emph{support} by
\[\mathrm{supp}(f):=\{x\in X:f(x)\neq0\}.\]
We also define the (at this stage possibly infinite valued) \emph{supremum norm} of $f$ by
\[\|f\|_\infty:=\sup_{x\in X}|f(x)|.\]
The bounded functions on $X$ will then be denoted by
\[\ell^\infty(X):=\{f:X\rightarrow\mathbb{C}\mid\|f\|_\infty<\infty\}.\]

If $X$ is a topological space, we also denote the \emph{open support} (respectively, \emph{closed support}) of $f:X\rightarrow\mathbb{C}$ by
\begin{align*}
    \mathrm{supp}^\circ(f)&:=\mathrm{int}(\mathrm{supp}(f)),\\
    \overline{\mathrm{supp}}(f)&:=\mathrm{cl}(\mathrm{supp}(f).
\end{align*}
We further denote the (compactly supported) continuous $\mathbb{C}$-valued functions on $X$ by
\begin{align*}
    C(X)&:=\{f:X\rightarrow\mathbb{C}\mid f\text{ is continuous}\} \text{ and }\\
    C_c(X)&:=\{f\in C(X):\overline{\mathrm{supp}}(f)\text{ is compact}\}.
\end{align*}
Note that $\mathrm{supp}^\circ(f)=\mathrm{supp}(f)$, for any $f\in C(X)$ -- from now on we usually write $\mathrm{supp}^\circ(f)$ rather than $\mathrm{supp}(f)$ to clearly distinguish open supports from closed supports.  Also note that $C_c(X)\subseteq\ell^\infty(X)$, so we can define the continuous functions \emph{vanishing at infinity} as the closure of $C_c(X)$ in $\ell^\infty(X)$.  We denote these by
\[C_0(X):=\mathrm{cl}_\infty(C_c(X)).\]

As a closed subspace of $\ell^\infty(X)$, we immediately see that $C_0(X)$ is a Banach space.  Under the usual pointwise product and conjugaction operations, it is also a commutative C*-algebra.  Indeed, the classic Gelfand duality tells us that every commutative C*-algebra $B$ is isomorphic to one of the form $C_0(X)$, for some locally compact Hausdorff space $X$.  Specifically, $B$ may be identified with $C_0(X_B)$, where $X_B$ is the set of all maximal ideals $I\subseteq B$ with the  hull-kernel topology, i.e. generated by $\{X_b\}_{b\in B}$ where
\[X_b=\{I\in X_B:b\notin I\}.\]
More precisely, for each $I\in X_B$, there is a unique character $\langle I\rangle:B\rightarrow\mathbb{C}$ with $\mathrm{ker}(\langle I\rangle)=I$ and each $b\in B$ can be identified with $b\in C_0(X_B)$ defined by $b(I)=\langle I\rangle(b)$.  Whenever convenient, we will use this identification of $B$ with $C_0(X_B)$ in our arguments as well as the continuous functional calculus it leads to.  Specifically, any continuous function $f$ on $\mathbb{C}$ with $f(0)=0$ can be applied to any normal element $a$ of a C*-algebra $A$ to yield another element $f(a)\in A$ identified with the function on $X_{C^*(a)}$ defined by $f(a)(I)=f(a(I))$ (where $C^*(a)$ denotes the C*-subalgebra of $A$ generated by $a$).

\subsection{Twists}

There are various definitions of twisted groupoids in the literature, e.g. see \cite{Armstrong2022, BFPR,Bonicke, CaH, Kumjian1986, MW92,Renault2008, Sims2020}.  For us it will be convenient to use a generalisation from \cite{Kumjian1986} of the original definition of a twist as a principal bundle of groupoids.

If $G$ is a groupoid, we let $G^{(0)}$ denote the unit space of $G$, $G^{(2)}$ denote the collection of composable pairs, and let $\mathsf{s} : G \to G^{(0)}$ and $\mathsf{r} : G \to G^{(0)}$ be the source and range maps, respectively. Recall that if $G$ is a groupoid, then $O\subseteq G$ is called a \emph{bisection} of $G$ if $O^{-1}O \subseteq G^{(0)}$ and $OO^{-1} \subseteq G^{(0)}$. A groupoid is \emph{\'{e}tale} when it carries a topology with a basis of open bisections which is closed under pointwise products and inverses. 

\begin{definition}\label{TwistDefinition}
    A \emph{$\mathbb{T}$-groupoid} is a Hausdorff topological groupoid $\Sigma$ on which we have a free continuous action of $\mathbb{T}$ such that, for all $t\in\mathbb{T}$ and $(e,f)\in\Sigma^{(2)}$,
    \[t(ef)=(te)f=e(tf).\]
    A \emph{twist} is a continuous open groupoid homomorphism $q:\Sigma\twoheadrightarrow G$ from a $\mathbb{T}$-groupoid $\Sigma$ onto a locally compact Hausdorff \'etale groupoid $G$ such that $\mathbb{T}$ acts transitively on each fibre (i.e. such that $q^{-1}(\{q(e)\})=\mathbb{T}e$, for all $e\in\Sigma$).
\end{definition}

The first thing to note is that twists restrict to homeomorphisms of unit spaces.

\begin{prop}\label{QuotientHomeo}
    Any twist $q:\Sigma\twoheadrightarrow G$ restricted to $\Sigma^{(0)}$ is a homeomorphism onto $G^{(0)}$.
\end{prop}

\begin{proof}
    Since $q$ is a groupoid homomorphism, we have $q(\Sigma^{(0)}) \subseteq \Go$.  For injectivity, suppose $e,f\in\Sigma^{(0)}$ with $q(e)=q(f)$.  Then $e=tf$, for some $t\in\mathbb{T}$, by the transitivity of the action on each fibre.  But then $\overline{t}e=\overline{t}(ee)=(\overline{t}e)e=fe$, so
    $\mathsf{s}(f)=\mathsf{r}(e)$ and hence $e=f$ as they are units.  Also, since $q$ is surjective, for any $g\in G^{(0)}$ we have $e\in\Sigma$ with $q(e)=g$ and hence $q(\mathsf{s}(e))=\mathsf{s}(q(e))=\mathsf{s}(g)=g$.  This shows that $q$ maps $\Sigma^{(0)}$ onto $G^{(0)}$.
    
    As $q$ is continuous on $\Sigma$, its restriction to $\Sigma^{(0)}$ is continuous.  It only remains to show that the restriction is also an open map.
   Accordingly, say $g_\lambda\rightarrow g$ in $G^{(0)}$.  As $q$ is an open map, we have a subnet $(g_\gamma)$ and another net $(e_\gamma)\subseteq\Sigma$ with $q(e_\gamma)=g_\gamma$, for all $\gamma$, and $e_\gamma\rightarrow q|_{\Sigma^{(0)}}^{-1}(g)$, by \cite[\S II.13.2]{DoranFell1988}.  For each $\gamma$, we then see that $\mathsf{s}(e_\gamma)\in\Sigma^{(0)}$ and $q(\mathsf{s}(e_\gamma))=\mathsf{s}(q(e_\gamma))=\mathsf{s}(g_\gamma)=g_\gamma$.  Also $\mathsf{s}(e_\gamma)\rightarrow \mathsf{s}(q|_{\Sigma^{(0)}}^{-1}(g))=q|_{\Sigma^{(0)}}^{-1}(g)$, showing $q|_{\Sigma^{(0)}}$ is an open map, again by \cite[\S II.13.2]{DoranFell1988}.
\end{proof}

\begin{remark}
    The proof of \cref{QuotientHomeo} applies equally well to `pretwists' $q:\Sigma\twoheadrightarrow G$, which are just like twists but without the requirement that $\Sigma$ is Hausdorff.  Indeed, we can use \cref{QuotientHomeo} to show that, for any pretwist $q:\Sigma\twoheadrightarrow G$ (over a locally compact Hausdorff \'etale groupoid $G$), the following are equivalent:
    \begin{enumerate}
        \item\label{SigmaHausdorff} $\Sigma$ is Hausdorff.
        \item\label{Sigma0closed} $\Sigma^{(0)}$ is closed in $\Sigma$.
        \item\label{qSigmaHausdorff} $q^{-1}(G^{(0)})$ is Hausdorff.
        \item\label{qSigma0closed} $\Sigma^{(0)}$ is closed in $q^{-1}(G^{(0)})$.
    \end{enumerate}
    To see this, first recall that a topological groupoid $\Gamma$ is Hausdorff precisely when its unit space $\Gamma^{(0)}$ is both Hausdorff and closed in $\Gamma$.  As $\Sigma^{(0)}$ is Hausdorff, taking $\Gamma$ to be $\Sigma$ and $q^{-1}(G^{(0)})$ respectively yields \eqref{SigmaHausdorff}$\Leftrightarrow$\eqref{Sigma0closed} and \eqref{qSigmaHausdorff}$\Leftrightarrow$\eqref{qSigma0closed} thanks to \cref{QuotientHomeo} and the assumption that $G^{(0)}$ is Hausdorff.  On the other hand, taking $\Gamma$ to be $G$ tells us that $G^{(0)}$ is closed and hence $q^{-1}(G^{(0)})$ is closed too, as $q$ is continuous, from which \eqref{Sigma0closed}$\Leftrightarrow$\eqref{qSigma0closed} immediately follows as well.  Thus in \cref{TwistDefinition}, we could replace the requirement that $\Sigma$ is Hausdorff with any of the other equivalent conditions above.
\end{remark}

Note the $\mathbb{T}$-action of any $\mathbb{T}$-groupoid $\Sigma$ is completely determined by its restriction to $\Sigma^{(0)}$, as $te=t(\mathsf{r}(e)e)=(t\mathsf{r}(e))e$, for all $t\in\mathbb{T}$ and $e\in\Sigma$.  For twists, this restricted $\mathbb{T}$-action yields a topological groupoid isomorphism from $\mathbb{T}\times G^{(0)}$ onto $q^{-1}(G^{(0)})$.

\begin{prop}
    If $q:\Sigma\twoheadrightarrow G$ is a twist, then we have a topological groupoid isomorphism $\iota:\mathbb{T}\times G^{(0)}\twoheadrightarrow q^{-1}(G^{(0)})$ given by $\iota(t,g)=tq|_{\Sigma^{(0)}}^{-1}(g)$ such that, for all $e\in\Sigma$,
    \[\iota(t,q(\mathsf{r}(e)))e=te=e\iota(t,q(\mathsf{s}(e))).\]
\end{prop}

\begin{proof}
    By \cref{QuotientHomeo}, $q|_{\Sigma^{(0)}}^{-1}$ is continuous.  As the $\mathbb{T}$-action is also continuous, $\iota$ is continuous.  As the $\mathbb{T}$-action is free, $\iota$ is also injective with inverse $\iota^{-1}(e)=(\tau(e),q(e))$ for $e \in q^{-1}(G^{(0)})$, where $\tau(e)$ is the unique element of $\mathbb{T}$ such that $e=\tau(e)\mathsf{s}(e)=\tau(e)\mathsf{r}(e)$.  If $\tau$ were not continuous then we would have a net $e_\lambda\rightarrow e$ in $q^{-1}(G^{(0)})$ such that $\tau(e_\lambda)\not\rightarrow\tau(e)$.  As $\mathbb{T}$ is compact, we would then have a subnet $(\tau(e_\gamma))$ converging to some $t\neq\tau(e)$.  But then $e_\gamma=\tau(e_\gamma)\mathsf{s}(e_\gamma)\rightarrow t\mathsf{s}(e)\neq\tau(e)\mathsf{s}(e)=e$, showing that the subnet $(e_\gamma)$ has two distinct limits, contradicting the fact $\Sigma$ is Hausdorff.  Thus, $\tau$ must be continuous and $\iota$ must be a homeomorphism, hence a groupoid isomorphism.  Finally just note that
    \[\iota(t,q(\mathsf{r}(e)))e=(tq|_{\Sigma^{(0)}}^{-1}(q(\mathsf{r}(e))))e=(t\mathsf{r}(e))e=t(\mathsf{r}(e)e)=te\]
    and, likewise, $te= e\iota(t,q(\mathsf{s}(e)))$, for all $t\in\mathbb{T}$ and $e\in\Sigma$.
\end{proof}

Sometimes twists are also required to be proper; however, for us this is automatic.

\begin{prop}
\label{lem:proper}
Every twist $q:\Sigma\twoheadrightarrow G$ is a proper map.
\end{prop}

\begin{proof}
We follow the proof of \cite[Lemma~2.2]{CDGanHV2024}.  Let $K\subseteq G$ be compact and $(e_\lambda)$ be a net in $q^{-1}(K)$.  As $K$ is compact, we may revert to a subnet if necessary to ensure $q(e_\lambda)\rightarrow q(e)\in K$, for some $e\in\Sigma$.  As $q$ is an open map, reverting to a further subnet if necessary, we have $(f_\lambda)\subseteq\Sigma$ with $q(e_\lambda)=q(f_\lambda)$ and $f_\lambda\rightarrow e$.  As the $\mathbb{T}$-action is transitive on fibres, we then have $(t_\lambda)\subseteq\mathbb{T}$ with $e_\lambda=t_\lambda f_\lambda$, for all $\lambda$.  Reverting to yet another subnet if necessary, we can ensure that $t_\lambda\rightarrow t$, for some $t\in\mathbb{T}$, and hence $e_\lambda=t_\lambda f_\lambda\rightarrow te$.  As $q(te)=q(e)\in K$, we have shown that every net in $q^{-1}(K)$ has a convergent subnet and hence $q^{-1}(K)$ is compact.  This shows $q$ is a proper map.
\end{proof}

It follows that the domain of a twist must be locally compact, just like the range.

\begin{cor}\label{lem:topprop on sigma}
If $q:\Sigma\twoheadrightarrow G$ is a twist, then $\Sigma$ is locally compact.
\end{cor}

\begin{proof}
    For every $e\in\Sigma$, the local compactness of $G$ means that we have some compact neighbourhood $K$ of $q(e)$.  As $q$ is a continuous proper map, $q^{-1}(K)$ is then a compact neighbourhood of $e$, showing that $\Sigma$ is indeed locally compact.
\end{proof}

\subsection{Twisted Groupoid C*-Algebras}\label{TwistedGroupoidCstarAlgebras}

Assume we have a twist $q:\Sigma\twoheadrightarrow G$.  We say $a:\Sigma\rightarrow\mathbb{C}$ is \emph{$\mathbb{T}$-contravariant} if $a(te)=\overline{t}a(e)$, for all $t\in\mathbb{T}$ and $e\in\Sigma$.  We define classes of $\mathbb{T}$-contravariant continuous $\mathbb{C}$-valued functions by
\begin{align*}
    C(\Sigma;G)&:=\{a\in C(\Sigma): a\text{ is $\mathbb{T}$-contravariant}\},\\
    C_0(\Sigma;G)&:=C(\Sigma;G)\cap C_0(\Sigma), \text{ and}\\
    C_c(\Sigma;G)&:=C(\Sigma;G)\cap C_c(\Sigma).
\end{align*}
The \emph{convolution} of any $a,b\in C_c(\Sigma;G)$ is the function $ab\in C_c(\Sigma;G)$ given by
\[ab(e):=\sum_{g\in q(e)G}a(\sigma(g))b(\sigma(g)^{-1}e)=\sum_{g\in Gq(e)}a(e\sigma(g)^{-1})b(\sigma(g)),\]
where $\sigma:G\rightarrow\Sigma$ is a (not necessarily continuous) section of $q$, i.e. satisfying $q(\sigma(g))=g$, for all $g\in G$. Note that compactness of supports is used to show the above sums are finite and define another element of $C_c(\Sigma;G)$.  Together with sums, scalar products with $z \in \mathbb{C}$ and the involution defined as usual by
\[(a+b)(e):=a(e)+b(e),\qquad (za)(e):=z(a(e))\qquad\text{and}\qquad a^*(e):=\overline{a(e^{-1})},\]
this makes $C_c(\Sigma;G)$ a *-algebra, that has a \emph{diagonal map} $D$ given by
\[D(a)(e):=\begin{cases}a(e)&\text{if }q(e)\in G^{(0)}\\0&\text{ otherwise.}\end{cases}\]
Note $D$ is an idempotent map from $C_c(\Sigma;G)$ onto
\[\Bc(\Sigma;G):=\{a\in C_c(\Sigma;G):q(\mathrm{supp}^\circ(a))\subseteq G^{(0)}\}.\]
And $\Bc(\Sigma;G)$ is contained in the *-semigroup
\[\Nc(\Sigma;G):=\{a\in C(\Sigma;G):q(\overline{\mathrm{supp}}(a))\text{ is a compact bisection}\},\]
which is in turn contained in the *-semigroup
\[S:=\{a\in C(\Sigma;G):q(\mathrm{supp}^\circ(a))\text{ is a bisection}\}\]
(note here that \cref{lem:proper} implies $\Nc(\Sigma;G)\subseteq S\cap C_c(\Sigma;G)$ and in general the containment can be proper).
We also define
\[N_0(\Sigma;G):=\{a \in C_0(\Sigma; G) : \: q(\supp^\circ(a)) \text{ is a bisection} \}.\]
By \cite[Theorem~3.1(3)]{ABCCLMR2023}, $G$ is effective if and only if the normalisers of $C_0(G^{(0)})$ in $C^*_r(\Sigma;G)$ are a subset of $S$.

\begin{prop}\label{NcNorm}
    Every C*-norm $\|\cdot\|$ on $C_c(\Sigma;G)$ agrees with $\|\cdot\|_\infty$ on $\Nc(\Sigma;G)$.
\end{prop}

\begin{proof}
The argument is similar to the untwisted case of \cite[Corollary~9.3.4]{Sims2020} and \cite[Proposition~3.14]{Exel2008}.  First, for any $n\in \Nc(\Sigma;G)$, \cref{lem:proper} yields $n^*n\in \Bc(\Sigma;G)$ and $\|n^*n\|_\infty=\|n\|^2_\infty$.  Thus, it suffices to show that every C*-norm $\|\cdot\|$ on $C_c(\Sigma;G)$ agrees with $\|\cdot\|_\infty$ on $\Bc(\Sigma;G)$. This follows for the same reason that any C*-norm on $C_c(G^{(0)})$ agrees with the supremum norm on $C_c(G^{(0)})$. Specifically note that, for any open $O\subseteq \Go$ with $\mathrm{cl}(O)$ compact,
    \[B_O=\{a\in C_c(\Sigma;G): q(\mathrm{supp}^\circ(a))\subseteq O\}\]
    is a C*-algebra with respect to $\|\cdot\|_{\infty}$ isomorphic to $C_0(O)$.
    As C*-norms on C*-algebras are unique (see \cite[II.2.2.10]{Blackadar2017}), it follows that $\|\cdot\|$ agrees with $\|\cdot\|_\infty$ on $B_O$.  But $B_c(\Sigma;G)$ is the union of these $B_O$'s and hence $\|\cdot\|$ agrees with $\|\cdot\|_\infty$ everywhere on $B_c(\Sigma;G)$.
\end{proof}

We will be particularly interested in $D$-contractive C*-norms on $C_c(\Sigma;G)$.  The completion of $C_c(\Sigma;G)$ with respect to such a norm is a C*-algebra $A$ on which we have a unique expectation $E$ extending the diagonal map $D$.  Any C*-algebra $A$ obtained in this way will be called a \emph{twisted groupoid C*-algebra}.  The two most important examples are the \emph{full} and \emph{reduced} twisted groupoid C*-algebras obtained as the completion of $C_c(\Sigma;G)$ with respect to the \emph{full norm} $\|\cdot\|_f$ and \emph{reduced norm} $\|\cdot\|_r$ respectively given by
\begin{align*}
    \|a\|_f&:=\sup\{\|a\|:\|\cdot\|\text{ is a C*-norm on }C_c(\Sigma;G)\} \text{ and }\\
    \|a\|_r&:=\sup\{\|D(c^*a^*ac)\|_\infty^{1/2}:c\in C_c(\Sigma;G)\text{ and }\|D(c^*c)\|_\infty\leq1\}.
\end{align*}
(For more information about why these are $D$-contractive C*-norms, see \cite{Armstrong2022}, \cite[Remark 2.5]{BFPR} or \cite{Sims2020}).

The reduced norm is the opposite of the full norm in the following sense.

\begin{prop}
    The reduced norm is the smallest $D$-contractive C*-norm on $C_c(\Sigma;G)$.
\end{prop}

\begin{proof}
    Let $\|\cdot\|$ be any $D$-contractive C*-norm on $C_c(\Sigma;G)$, which must agree with $\|\cdot\|_\infty$ on $\Bc(\Sigma;G)$ by \cref{NcNorm}.  By \cite[II.6.10.2]{Blackadar2017}, the unique expectation $E$ on the completion $A$ extending $D$ must be positive and hence, for all $a,c\in C_c(\Sigma;G)$ with $\|D(c^*c)\|_\infty\leq1$,
    \[\|D(c^*a^*ac)\|_\infty=\|E(c^*a^*ac)\|\leq\|E(\|a\|^2c^*c)\|=\|a\|^2\|E(c^*c)\|=\|a\|^2\|D(c^*c)\|_\infty\leq\|a\|^2.\]
    As $c$ was arbitrary, this shows that $\|a\|_r\leq\|a\|$.
\end{proof}

The $D$-contractive condition is crucial here, as there may be no minimal C*-norm, even for examples like the trivial twist on $\mathbb{Z}$ -- see Caleb Eckhardt's comment in \cite{Daws2014}.
While a C*-completion $A$ of $C_c(\Sigma;G)$ will no longer just consist of functions on $\Sigma$, if the C*-norm in question is $D$-contractive then there will at least be a canonical `$j$-map' taking elements of $A$ back to $\mathbb{T}$-contravariant functions on $\Sigma$.

\begin{prop}
    If $A$ is any $D$-contractive C*-completion of $C_c(\Sigma;G)$ then there is unique contractive map $j:A\rightarrow C_0(\Sigma;G)$ extending the identity on $C_c(\Sigma;G)$.
\end{prop}

\begin{proof}
    It suffices to show that every $D$-contractive C*-norm $\|\cdot\|$ on $C_c(\Sigma;G)$ dominates the supremum norm $\|\cdot\|_\infty$.  Accordingly, take any $a\in C_c(\Sigma;G)$ and $e\in\Sigma$.  Further taking any $n\in \Nc(\Sigma;G)$ with $n(e)=1=\|n\|_\infty$, note $\|n\|=1$ by \cref{NcNorm} so
    \[|a(e)|=|an^*(\mathsf{r}(e))|\leq\|D(an^*)\|_\infty=\|D(an^*)\|\leq\|a\|\|n\|=\|a\|.\]
    As $e$ was arbitrary, this shows that $\|a\|_\infty\leq\|a\|$, as required.
\end{proof}
  
If $A$ above is the reduced C*-algebra $C_r^*(\Sigma;G)$ then $j$ will even be injective and $j(ab)$ will always be the convolution of $j(a)$ and $j(b)$ (where the sums involved may no longer be finite but still converge absolutely -- see \cite{BFPR}).  In other words, $j(C_r^*(\Sigma;G))$ is a concrete C*-algebra of functions with respect to convolution which we could simply identify with $C_r^*(\Sigma;G)$. Reduced C*-algebras can even be defined directly from functions in the first place, even for more general Fell bundles, as in \cite{Bice2024}.  However, in keeping with the more traditional mindset, in the present paper we will continue to distinguish $C_r^*(\Sigma;G)$ from $j(C_r^*(\Sigma;G))$ and view the latter primarily as a linear subspace of $C_0(\Sigma;G)$.

In general, we can still show that the map $j$ is injective on the closure of $\Nc(\Sigma;G)$. 

\begin{prop}\label{jSemigroupIso}
    If $A$ is any $D$-contractive C*-completion of $C_c(\Sigma;G)$, then the $j$-map restricts to a semigroup isomorphism from $\mathrm{cl}(\Nc(\Sigma;G))$ onto $N_0(\Sigma;G)$.
\end{prop}

\begin{proof}
    As $j$ is contractive, $j(\mathrm{cl}(\Nc(\Sigma;G)))\subseteq\mathrm{cl}_\infty(j(\Nc(\Sigma;G)))=N_0(\Sigma;G)$. For the reverse inclusion note that, for any $n\in N_0(\Sigma;G)$ and $k\in\mathbb{N}$, we can define $n^k\in \Nc(\Sigma;G)$ with $\|n-n^k\|_\infty\leq1/k$ by
    \[n^k(e):=\tfrac{\max(0,|n(e)|-\tfrac{1}{k})}{|n(e)|}n(e).\]
    It follows that $\|n^l-n^k\|= \|n^l-n^k\|_\infty\leq|\frac{1}{l}-\frac{1}{k}|$ so $(n^k)$ is Cauchy and hence has some limit $m\in\mathrm{cl}(\Nc(\Sigma;G))$ with $j(m)=n$, thus showing that $N_0(\Sigma;G)\subseteq j(\mathrm{cl}(\Nc(\Sigma;G)))$.  For any other $(m_k)\subseteq \Nc(\Sigma;G)$ with $m_k\rightarrow n$ in $N_0(\Sigma;G)$, we can revert to a subsequence if necessary to ensure that $\|m_k-n\|_\infty<1/k$.  Then $\overline{\mathrm{supp}}(m_k^k)\subseteq\mathrm{supp}^\circ(n)$ and hence $m_k^k-n^k\in \Nc(\Sigma;G)$, so
    \begin{align*}
        \|m_k-n^k\|&\leq\|m_k-m_k^k\|+\|m_k^k-n^k\|\\
        &=\|m_k-m_k^k\|_\infty+\|m_k^k-n^k\|_\infty\\
        &\leq2\|m_k-m_k^k\|_\infty+\|m_k-n\|_\infty+\|n-n^k\|_\infty\\
        &\leq4/k\\
        &\rightarrow0.
    \end{align*}
    Thus $m_k\rightarrow m$ in $A$, showing that $m$ is the unique element of $\mathrm{cl}(\Nc(\Sigma;G))$ with $j(m)=n$.

    Finally, for any $m,n\in\mathrm{cl}(\Nc(\Sigma;G))$, we see that $m^kn^k$ converges to both $j(mn)$ and $j(m)j(n)$ in $N_0(\Sigma;G)$, so $j$ is not just bijective but also a semigroup isomorphism.
\end{proof}

Accordingly, we will often identify $\mathrm{cl}(N_c(\Sigma;G))$ with $N_0(\Sigma;G)$ via \cref{jSemigroupIso} and refer to both as the \emph{monomial semigroup}.
 The terminology here comes from the fact that, in the special case of a matrix algebra $M_k$ (viewed as $C(\Sigma;G)$ where $\Sigma$ is the trivial twist over the full equivalence relation $G=\{(i,j):1\leq i,j\leq k\}$ on $k$ elements), the monomial semigroup consists precisely of the monomial matrices, i.e.~those matrices with at most one non-zero entry in each row and each column.

Incidentally, building on \cref{jSemigroupIso}, one can even show that the $j$-map restricts to a *-algebra isomorphism from $\mathrm{span}(\mathrm{cl}(\Nc(\Sigma;G)))$ onto $\mathrm{span}(N_0(\Sigma;G))$.  This will also become apparent as a by-product of our later work.

\section{Cartan Semigroups}\label{sec:Cartansemi}

For any subset $N$ of a C*-algebra $A$, let us denote its positive cone by
\[N_+:=\{n^*n:n\in N\}.\]
Furthermore, let $C^*(N)$ denote the C*-subalgebra generated by $N$.

\begin{definition}\label{def:semiCartan}
Let $A$ be a C*-algebra.  We call $N\subseteq A$ a \emph{Cartan semigroup} if
\begin{enumerate}
\item $N$ is a closed *-subsemigroup of $A$ with dense span,
\item\label{item:CommutativePositiveCone} $B:= C^*(N_+)$ is a commutative subsemigroup of $N$, and
\item\label{item:faithfulbistable} there is an expectation $E$ from $A$ onto $B$ such that, for all $n\in N$,
\begin{equation}\label{StarStability}
    \tag{Stable}E(n)n^*\in B.
\end{equation}
\end{enumerate}
In this case we call $B$ the associated \emph{semi-Cartan subalgebra}.
\end{definition}

Cartan semigroups and their associated semi-Cartan subalgebras provide a convenient general framework to extend the Kumjian-Renault theory, as we will see in the following sections. The first thing to note is that semi-Cartan subalgebras (even those coming from the summable Cartan semigroups defined below) are more general than Cartan subalgebras.  For one thing, a semi-Cartan subalgebra $B$ needs only be commutative, not maximal commutative (from here on abbreviated to MASA as usual).  As a result, semi-Cartan subalgebras arise even from non-effective groupoids (e.g. discrete groups).  But we also do not require the expectation $E$ to be faithful, and consequently semi-Cartan subalgebras arise not only in reduced C*-algebras but also in full C*-algebras and all other C*-completions of $C_c(\Sigma;G)$ in between.

The Cartan semigroups that arise in practice often satisfy one further condition.

\begin{definition} Given a Cartan semigroup $N$ with associated semi-Cartan subalgebra $B$, call $C\subseteq N$ a \emph{compatible} subset of $N$ if, for all $c,d\in C$, both $c^*d$ and $cd^*$ lie in $B$.  We call $N$ \emph{summable} if it is closed under taking finite (equivalently pairwise) compatible sums, i.e. if, for all $m,n\in N$,
\[\tag{Summable}m^*n,mn^*\in B\qquad\Rightarrow\qquad m+n\in N.\]
\end{definition}

\begin{prop}\label{TwistedGroupoid->CartanSemigroup} Let $q: \Sigma \twoheadrightarrow G$ be a twist. 
    If $A$ is a $D$-contractive C*-completion of $C_c(\Sigma;G)$, then the monomial semigroup $N:=\mathrm{cl}(\Nc(\Sigma;G))$ is a summable Cartan semigroup in $A$ whose associated semi-Cartan subalgebra is the diagonal $B:=\mathrm{cl}(\Bc(\Sigma;G))$.
\end{prop}

\begin{proof}
    First note that, as $\Nc(\Sigma;G)$ is a *-subsemigroup of $C_c(\Sigma; G)$ that spans $C_c(\Sigma;G)$, its closure $N$ is a *-subsemigroup with dense span.  Also $\Bc(\Sigma;G)\subseteq \Nc(\Sigma;G)$ immediately implies $B\subseteq N$.  Now identify $N$ with $N_0(\Sigma;G)$, which we can do by \cref{jSemigroupIso}.  For any $n\in N$, setting $O=q(\mathrm{supp^\circ}(n))$ yields $q(\mathrm{supp^\circ}(n^*n))\subseteq O^{-1}O\subseteq G^{(0)}$.  Thus $N_+\subseteq B$ and hence $C^*(N_+)\subseteq B$, while certainly $B_+\subseteq N_+$ and hence $B=C^*(B_+)\subseteq C^*(N_+)$ so $B=C^*(N_+)$.  Again, for any $n\in N$, set $O=q(\mathrm{supp^\circ}(n))$ and note that
    \[q(\mathrm{supp^\circ}(E(n)n^*))\subseteq(O\cap G^{(0)})O^{-1}\subseteq OO^{-1}\subseteq G^{(0)}\]
    so $E(n)n^*\in B$, thus verifying \eqref{StarStability}.  This shows that $N$ is a Cartan semigroup.  For summability, take $m,n\in N$ with $m^*n,mn^*\in B$.  Letting $U=q(\mathrm{supp}^\circ(m))$ and $V=q(\mathrm{supp}^\circ(n))$, note $G^{(0)}\supseteq q(\mathrm{supp}^\circ(m^*n))=U^{-1}V$ and, likewise, $UV^{-1}\subseteq G^{(0)}$.  This means $U\cup V\supseteq q(\mathrm{supp}^\circ(m+n))$ is a bisection and hence $m+n\in N$, as required.
\end{proof}

If $G$ is not effective then the semi-Cartan $B$ is not a MASA.  Indeed, then we have some $a\in C_c(\Sigma;G)$ with $q(\mathrm{supp}^\circ(a))$ contained in the isotropy of $G$ but not contained in the unit space $G^{(0)}$.  This implies that $a$ commutes with every $b\in B$ even though $a$ itself is not in $B$, showing that $B$ is not a MASA.

\begin{remark}\label{SpecialBasisExample}
    There can also be non-summable Cartan semigroups contained in $\mathrm{cl}(\Nc(\Sigma;G))$, for example obtained by restricting to functions supported on a suitable basis.  Specifically, say we have basis $\mathcal{B}$ of open bisections of $G$ containing the unit space that is also closed under open subsets, products and inverses, i.e.
    \begin{enumerate}
        \item\label{G0inB} $G^{(0)}\in\mathcal{B}$,
        \item\label{SubsetClosed} $O\subseteq U\in\mathcal{B}$ implies $O\in\mathcal{B}$, and
        \item\label{InverseSemigroup} $OU,O^{-1}\in\mathcal{B}$, for all $O,U\in\mathcal{B}$.
    \end{enumerate}
    Then we have a Cartan semigroup given by
    \[N_\mathcal{B}:=\mathrm{cl}(\{n\in C_c(\Sigma;G):q(\mathrm{supp}^\circ(n))\in\mathcal{B}\}).\]
    Indeed, \eqref{InverseSemigroup} implies that $N_\mathcal{B}$ is a *-semigroup while \eqref{G0inB} and \eqref{SubsetClosed} imply that $C^*(N_{\mathcal{B}_+})\subseteq N_\mathcal{B}$.  As $\mathcal{B}$ is a basis, $N_\mathcal{B}$ still has dense span in $A$ and the other required properties follow from an argument similar to the proof of \cref{TwistedGroupoid->CartanSemigroup}.

    For an example of this, consider the trivial twist $\Sigma=G\times\mathbb{T}$ over the discrete principal groupoid $G=\{(i,j):1\leq i,j\leq 2\}$ coming from the full equivalence relation on $2$ elements, and let
    \[\mathcal{B}=\{O:O\subseteq G^{(0)}\}\cup\{\{(1,2)\},\{(2,1)\}\}.\]
    Then $N:=\mathrm{cl}(\Nc(\Sigma;G))=\{a \in C(\Sigma;G) : \: q(\supp^\circ(a)) \text{ is a bisection} \} \approx \{a \in C(G) : \: \supp^\circ(a) \text{ is a bisection} \}$ is isomorphic to the multiplicative semigroup of $2\times2$ matrices with at most one non-zero entry in each row and column.  In particular, the off-diagonal matrix with $1$ in both the top right and bottom left corner will be in $N$ but not in $N_\mathcal{B}$.
\end{remark}

\cref{TwistedGroupoid->CartanSemigroup} above of course applies to trivial twists, in which case $C_c(\Sigma;G)$ can be identified with $C_c(G)$.  When $G$ is a discrete group (e.g.~the integers $\mathbb{Z}$), the monomial semigroup is then just $\bigcup_{g\in G}\mathbb{C}\delta_g$ with diagonal $\mathbb{C}\delta_e$, where $e\in G$ is the identity. 

Here is another example of a Cartan semigroup, which is really just a special case of \cref{TwistedGroupoid->CartanSemigroup} in disguise.

\begin{example}
\label{ex:twist}
Let $A$ be a C*-algebra, and suppose there is a state $\phi:A\rightarrow\mathbb{C}$ and a unitary $u\in A$ generating $A$ such that $\phi(u^k)=0$, for all non-zero $k\in\mathbb{Z}$.  Then $N=\bigcup_{k\in\mathbb{Z}}\mathbb{C}u^k$ is a Cartan semigroup with $B=C^*(N_+)=\mathbb{C}1$ and $E(a)=\phi(a)1$, for all $a\in A$. To see that Condition~\eqref{item:faithfulbistable}~\eqref{StarStability} holds, note that $E(n)n^*\in\mathbb{R}_+1$ whenever $n\in\mathbb{C}1$ and $E(n)n^*=0$ for all other $n\in N$.  If $\phi$ is also faithful, then it will follow from the general theory developed in the following sections and especially \cref{cor:it} that $A$ must be isomorphic to $C^*_r(\mathbb{Z})$ (which is isomorphic to $C(\mathbb{T}))$.
\end{example}

As mentioned in the introduction, there has been work by others that have recovered a twisted groupoid from a dual group action rather than using effectiveness \cite{BFPR}. From their work, we can see another example of a Cartan semigroup. 

\begin{example}
Let $A$ be a C*-algebra topologically graded by a discrete abelian group $\Gamma$ where its dual group $\hat{\Gamma}$ acts strongly on $A$, and let $D$ be an abelian C*-subalgebra of $A_0$ such that $(A,D)$ is $\Gamma$-Cartan, as in \cite{BFPR}. Let $N_h(A,D)$ be the so-called \emph{homogeneous normalisers}, which is a subset of the normalisers. Then $N_h(A,D)$ is a closed *-semigroup of $A$, $N_h(A,D)$ has dense span due to \cite[Lemma~3.10(1)]{BFPR}, and $D=C^*((N_h(A,D))_+)$ due to \cite[Lemma~3.10(3)]{BFPR}. If $E$ is the expectation from the Cartan pair $(A_0,D)$ and $\Phi_0$ is as given, $E \circ \Phi_0$ is an expectation from $A$ onto $D$; it can be checked to be stable by \cite[Lemma~3.4]{BFPR}.
    \end{example}

When $N$ is a Cartan semigroup with semi-Cartan subalgebra $B=C^*(N_+)$, note
\[B_+=N_+=N\cap A_+.\]
Indeed, $B\subseteq N$ by definition so $B_+\subseteq N_+$, while if $n\in N$ then $\sqrt{n^*n}\in B$ so $N_+\subseteq B_+$.  Similarly, the second equality follows from the fact that $N_+\subseteq N\cap A_+$ because $N$ is a *-semigroup, while if $a\in N\cap A_+$ then $a^2\in N_+\subseteq B$ so $a=\sqrt{a^2}\in B_+=N_+$.

Another observation is the following. 

\begin{lemma}\label{prop:approxunit}
    Every semi-Cartan subalgebra $B$ contains an approximate unit for $A$.
\end{lemma}

\begin{proof}
    We argue as in \cite[Lemma 3.10 (2)]{BFPR}.  If $N$ is a Cartan semigroup with $B=C^*(N_+)$, then any approximate unit for $B$ is an approximate unit for any $n\in N$, as $n^*n,nn^*\in B$ (see \cite[Equation~(3.12)]{BFPR}).  This then extends to $\mathrm{span}(N)$ and its closure $A$.
\end{proof}

From now on, it will be convenient to fix a sequence of non-zero polynomials $(p_k)$ with zero constant terms that converge to $1$ uniformly on all compact subsets of $\mathbb{R}\setminus\{0\}$ (such $(p_k)$ exist by the Stone-Weierstrass theorem).  For any C*-algebra $A$ and any $a\in A$, it then follows that $(p_k(aa^*))$ and $(p_k(a^*a))$ are left and right approximate units for $a$, i.e.
\begin{equation} \label{def:p_k poly} p_k(aa^*)a=ap_k(a^*a)\rightarrow a.\end{equation}
Replacing each $p_k$ with $p_k^2/\max_{x\in[0,1]}p_k^2(x)$ if necessary, we may further assume that $p_k(\mathbb{R})=\mathbb{R}_+$ and $p_k([0,1])=[0,1]$, so that $p_k(a)$ is in $A_+$ or $A^1_+$ whenever $a$ is.  As $p_k$ has zero in both its constant term and the coefficient of its $x$ term, there is a unique polynomial $q_k$ with
\[p_k(x)=xq_k(x),\]
which again has zero constant term and satisfies $q_k(\mathbb{R}_+)=\mathbb{R}_+$

Any Cartan pair (defined in the beginning of the introduction) is isomorphic to the reduced C*-algebra of a twist over an effective groupoid together with its diagonal subalgebra (see \cite[Theorem~5.9]{Renault2008} for the second countable case and \cite[Theorem~1.1]{Raad2022} for the general result). Thus Cartan subalgebras are semi-Cartan, using \cref{TwistedGroupoid->CartanSemigroup} and $\mathrm{cl}(B_c(\Sigma;G))= C_0(q^{-1}(G^{(0)});G^{(0)})$.  We can also prove this directly as follows.  

\begin{prop}\label{prop:CartanGeneralisation}
    Suppose $(A,B)$ satisfies all properties of a Cartan pair except that the expectation $E:A\twoheadrightarrow B$ need not be faithful.  Then the entire normaliser semigroup
    \[N(B)=\{n\in A:n^*Bn\cup nBn^*\subseteq B\}\]
    forms a Cartan semigroup with associated semi-Cartan subalgebra $B=C^*(N(B)_+)$.
\end{prop}

\begin{proof}
    All the required conditions on $N(B)$ are immediate except for \eqref{StarStability} in Condition \eqref{item:faithfulbistable}.  To see that \eqref{StarStability} also holds first note that, for all $b\in B$ and $n\in N(B)$,
    \[E(n)n^*nn^*b=E(n)n^*bnn^*=E(nn^*bn)n^*=E(bnn^*n)n^*=bE(n)n^*nn^*.\]
    Likewise, $E(n)n^*(nn^*)^kb=bE(n)n^*(nn^*)^k$ for all $k>1$ and hence
    \begin{equation}\label{Commutant}
    E(n)n^*b=\lim_kE(n)n^*p_k(nn^*)b=\lim_kbE(n)n^*p_k(nn^*)=bE(n)n^*.
    \end{equation}
    As $B$ is a MASA, it follows that $E(n)n^*\in B$, showing that $N(B)$ is a Cartan semigroup.
    Now $B\subseteq N(B)$ and hence $B_+\subseteq N(B)_+$, while $N(B)_+\subseteq B\cap A_+=B_+$ because $B$ contains an approximate unit for $A$.  Thus $B=C^*(B_+)=C^*(N(B)_+)$.
\end{proof}

We have already noted above that the converse to \cref{prop:CartanGeneralisation} is false, i.e.~there are semi-Cartan subalgebras that are not MASAs and hence not Cartan subalgebras.  Indeed, our primary goal is to show that semi-Cartan subalgebras are precisely the right generalisation needed to extend the Kumjian-Renault theory to non-effective groupoids and non-reduced C*-completions.  From this it will follow that summable Cartan semigroups really characterise the situation in \cref{TwistedGroupoid->CartanSemigroup}.  In fact, even when a C*-algebra $A$ contains a (potentially non-summable) Cartan semigroup $N$ then we can still show it is isomorphic to a twisted groupoid C*-algebra where the closure of the compatible sums of $N$ corresponds precisely to the closure of the functions supported on a open bisections with compact closure.

\section{Algebraic Properties}\label{sec:algprop}
In this section, we exhibit several useful algebraic properties of a Cartan semigroup $N$ and its associated semi-Cartan subalgebra $B$.  To avoid constantly repeating our basic assumptions, let us assume throughout the rest of the paper (except in \cref{cor:masa} where our slightly different assumptions on $A$, $B$ and $E$ are stated explicitly) that
\begin{center}
    $N$\textbf{ is a Cartan semigroup in a C*-algebra }$A$\textbf{ with stable expectation }$E:A\twoheadrightarrow B:=C^*(N_+)$.
\end{center}

We start by showing that $N$ is closed under scalar products and contained in the normaliser semigroup of $B$.  By definition, $N$ also contains $B$ and has dense span, and is thus a `skeleton' of $(A,B)$, in the sense of \cite[Definition 1.8]{Pitts2017} and \cite[Definition 2.4]{Pitts2021}.

\begin{lemma}\label{lem:Normal}
    For all $\alpha\in\mathbb{C}$ and $n\in N$, $\alpha B\cup n^*Bn\subseteq B$, that is,
    \begin{equation*}
        \mathbb{C}N=N\subseteq N(B).
    \end{equation*}
\end{lemma}

\begin{proof}
    Let $(b_\lambda)\subseteq B$ be an approximate unit for $A$.  Then, for all $\alpha\in\mathbb{C}$ and $n\in N$,
    \[\alpha n=\lim_\lambda \alpha b_\lambda n\in\mathrm{cl}(BN)\subseteq\mathrm{cl}(NN)\subseteq N.\]
    Next, for any $m,n\in N$, note $n^*m^*mn=(mn)^*mn\in N_+$.  So $n^*N_+n\subseteq N_+$ and hence
    \begin{align*}n^*Bn&=n^*\mathrm{span}(B_+)n=n^*\mathrm{span}(N_+)n=\mathrm{span}(n^*N_+n)\\&\subseteq\mathrm{span}(N_+)=\mathrm{span}(B_+)=B.\qedhere\end{align*}
\end{proof}

We can replace $n^*$ in \cref{lem:Normal} by any $m\in N$ with $mn\in B$, showing that $B$ is \emph{binormal}.  Binormality is one of the required conditions for $(N,B,B)$ to be a structured semigroup in \cite[Definition 1.6]{Bice2022}.  The theory in \cite{Bice2022} then gives us an \'etale groupoid of ultrafilters that we can use as a generalised Weyl groupoid in the following sections.

\begin{cor}
\label{cor:binormal}
    For all $m,n\in N$,
    \begin{equation}\label{Binormal}
        \tag{Binormal}mn\in B\qquad\Rightarrow\qquad mBn\subseteq B.
    \end{equation}
\end{cor}

\begin{proof}
    For any $b\in B$ and $m,n\in N$ with $mn\in B$, \cref{lem:Normal} yields
    \[p_k(mm^*)mbn=mp_k(m^*m)bn=mbp_k(m^*m)n\in Bmn\subseteq BB\subseteq B,\]
    for all $k$.  Thus $mbn=\lim_kp_k(mm^*)mbn\in\mathrm{cl}(B)\subseteq B$, showing that $mBn\subseteq B$.
\end{proof}

We can strengthen \cref{lem:Normal} above from $N\subseteq N(B)$ to $N=N(B)$ when $B$ is a MASA and $N$ is summable -- see \cref{MASAimpliesNormalisers}.  However, in general the inclusion can be strict and any given semi-Cartan subalgebra $B\subseteq A$ can be associated to multiple Cartan semigroups $N\subseteq N(B)$ with $B= C^*(N_+)$.

For normaliser semigroups of Cartan subalgebras, summability comes for free.

\begin{prop}\label{SummableNormalisers}
    The normaliser semigroup $N(B)$ of a Cartan subalgebra $B$ in \cref{prop:CartanGeneralisation} is summable.
\end{prop}

\begin{proof}
    In \cref{prop:CartanGeneralisation} it was already shown that $N(B)$ is a Cartan semigroup with associated semi-Cartan subalgebra $B$.  Thus $mBn\subseteq B$, for all $m,n\in N(B)$ with $mn\in B$, by \eqref{Binormal}. So if $m,n\in N(B)$ and $m^*n,mn^*\in B$ then
    \[(m+n)^*B(m+n)\subseteq m^*Bm+m^*Bn+n^*Bm+n^*Bn\subseteq B.\]
    Likewise $(m+n)B(m+n)^*\subseteq B$, showing that $m+n\in N(B)$, as required. 
\end{proof}

In \S8.1 of \cite{Bice2022}, some stronger results are proved for \emph{symmetric} structured semigroups, which also apply to all Cartan semigroups, by the following result.

\begin{lemma} \label{lem:Symmetry}
For all $l,m\in N$,
\begin{equation}\label{Symmetry}
\tag{Symmetry}lm\in B\qquad\Rightarrow\qquad mlml\in B.
\end{equation}
\end{lemma}

\begin{proof}
If $lm\in B$ then note that, as $l^*l,mm^*\in B$ commute,
\[mll^*lmm^*ml=mlmm^*l^*lml\in mBm^*l^*Bl\subseteq BB\subseteq B.\]
Likewise, $ml(p_k(l^*l))(p_k(mm^*))ml\in B$ so $mlml=\lim_kml(p_k(l^*l))(p_k(mm^*))ml\in B$.
\end{proof}

Another observation that will be useful later in \eqref{StarSwitch} is the following.

\begin{lemma}\label{mnstar}
For all $m,n\in N$, if $m=mn$ then $m=mn^*$.
\end{lemma}

\begin{proof}
For all $a,b\in A$, note that $a=ab$ is equivalent to $a^*a=a^*ab$.  Indeed, $a=ab$ certainly implies $a^*a=a^*ab$, while conversely if $a^*a=a^*ab$ then
\[(a-ab)^*(a-ab)=a^*a-a^*ab-b^*a^*a+b^*a^*ab=a^*a-a^*a-a^*a+a^*a=0\]
so $a-ab=0$ and hence $a=ab$.  In particular, if $m=mn$ then $|m|^2=m^*m=m^*mn=|m|^2n$ and hence $|m|=|m|n=n^*|m|$ (where $|m|=\sqrt{m^*m})$.  It then follows that $m^*m=|m|^2=|m|nn^*|m|=m^*mnn^*$, as $|m|,n^*n\in B$ commute, so $m=mnn^*=mn^*$.
\end{proof}

\section{The Restriction Relation}
In this section, we introduce the first of two transitive relations we study on $N$ and prove several properties we will need later when we examine the expectation $E$ and the ultrafilter groupoid $G$ (e.g.~see \cref{prop:resE}, \cref{lem:units and Un}, and \cref{prop:G is Hausdorff}).

Let us call $m\in N$ a \emph{restriction} of $n\in N$ if we have a sequence $(b_k)\subseteq B$ with
\[m=\lim_kmb_k=\lim_knb_k.\]
Put another way this means that, for all $\varepsilon>0$, we have $b\in B$ with $\|m-mb\|,\|m-nb\|<\varepsilon$. 
 The corresponding restriction relation will be denoted by $\sqsubseteq$, i.e.
\[m\sqsubseteq n\qquad\Leftrightarrow\qquad m\text{ is a restriction of }n.\]
The terminology here comes from the following characterisation of $\sqsubseteq$ in the situation of \cref{TwistedGroupoid->CartanSemigroup} of twisted groupoid C*-algebras and their monomial semigroups (see \cite[(9.1)]{Bice2023} for an analogous characterisation of a purely algebraic version of $\sqsubseteq$).

\begin{prop}
    If $A$ is a $D$-contractive C*-completion of $C_c(\Sigma;G)$, for some twist $q:\Sigma\twoheadrightarrow G$, and $N=\mathrm{cl}(\Nc(\Sigma;G))$ $($so $B=\mathrm{cl}(\Bc(\Sigma;G)))$ then, for all $m,n\in N$,
    \[m\sqsubseteq n\qquad\Leftrightarrow\qquad j(m)|_{\supp^\circ(j(m))}=j(n)|_{\supp^\circ(j(m))}.\]
    
\end{prop}
\begin{proof}
    By \cref{jSemigroupIso}, we can identify $N$ with $N_0(\Sigma;G)$ and omit the $j$'s.
    Assume $m,n\in N$ and $m|_{\supp^\circ(m)}=n|_{\supp^\circ(m)}$.  Letting $(b_k)\subseteq B$ be any bounded sequence converging to $0$ uniformly on $\Go\setminus \mathsf{s}(\supp^\circ(m))$ and converging to $1$ uniformly on all compact subsets of $\mathsf{s}(\supp^\circ(m))$, we see that $m=\lim_kmb_k=\lim_knb_k$ so $m\sqsubseteq n$.
    
    Conversely, if $m\sqsubseteq n$ then we have $(b_k)\subseteq B$ with $m=\lim_kmb_k=\lim_knb_k$.  For all $e\in\supp^\circ(m)$, this means $m(e)=\lim_km(e)b_k(\mathsf{s}(e))$ and hence $b_k(\mathsf{s}(e))\rightarrow1$ so
    \[n(e)=\lim_kn(e)b_k(\mathsf{s}(e))=\lim_km(e)b_k(\mathsf{s}(e))=m(e).\]
    As $e$ was arbitrary, this shows that $m|_{\supp^\circ(m)}=n|_{\supp^\circ(m)}$.
\end{proof}

Next we show that we can choose the sequence $(b_k)$ to lie in the positive unit ball.

\begin{lemma}\label{PositiveBallSequence}
If $m\sqsubseteq n$, then we can pick $(b_k)\subseteq B^1_+$ with $m=\lim_kmb_k=\lim_knb_k$.
\end{lemma}

\begin{proof}
    Take $(b_k)\subseteq B$ with $m=\lim_kmb_k=\lim_knb_k$ and define $f$ on $\mathbb{C}$ by
    \[f(x)=\begin{cases}|x|&\text{if }|x|\leq1\\1&\text{if }|x|\geq1.\end{cases}\]
    As $|1-|z||\leq|1-z|$, for all $z\in\mathbb{C}$, it follows that
    \[\|m(1-f(b_k))\|^2=\|m(1-f(b_k))^2m^*\|\leq\|m(1-b_k)(1-b_k^*)m^*\|=\|m(1-b_k)\|^2\rightarrow 0,\]
    so $m(1-f(b_k))\rightarrow0$ and hence $m=\lim_kmf(b_k)$.  Likewise,
    \[\|(m-n)f(b_k)\|\leq\|(m-n)b_k\|\rightarrow 0\]
    so $\lim_kmf(b_k)=\lim_knf(b_k)$. As $(f(b_k))\subseteq B^1_+$, we are done.
\end{proof}

In the definition of a restriction, we could also have put the $b_k$'s on the other side.  To prove this, first note that
\begin{equation}\label{LimitLemma}
    a_k\rightarrow a\quad\text{and}\quad ab_k\rightarrow a\qquad\Rightarrow\qquad a_kb_k\rightarrow a,
\end{equation}
whenever $a\in A$, $(a_k)\subseteq A$ and $(b_k)\subseteq A^1_+$, as we then see that
\[\|a_kb_k-a\|=\|a_kb_k-ab_k+ab_k-a\|\leq\|a_k-a\|\|b_k\|+\|ab_k-a\|\rightarrow0.\]

\begin{lemma}
    For any $m,n\in N$,
    \[m\sqsubseteq n\qquad\Leftrightarrow\qquad \text{there exists }(b_k)\subseteq B^1_+ \text{ such that }m=\lim_kb_km=\lim_kb_kn.\]
\end{lemma}

\begin{proof}
    Say $m\sqsubseteq n$ so we have $(c_k)\subseteq B^1_+$ with $m=\lim_kmc_k=\lim_knc_k$.  For each $k$, let
    \[b_k=mc_km^*(q_k(mm^*)).\]
    By \cref{lem:Normal}, $mc_km^*\in B$ and $b_k\leq mm^*q_k(mm^*)=p_k(mm^*)\leq1$ so $b_k\in B^1_+$.  Also $m^*n=\lim_kc_k^*n^*n\in\mathrm{cl}(BB)\subseteq B$ so, as $nc_k\rightarrow m$ and $p_k(mm^*)m\rightarrow m$, \eqref{LimitLemma} yields
    \[b_kn=mc_kq_k(m^*m)m^*n=mq_k(m^*m)m^*nc_k=p_k(mm^*)nc_k\rightarrow m.\]
    Replacing $n$ above with $m$ yields $b_km\rightarrow m$ as well, finishing the proof of the $\Rightarrow$ part.  The $\Leftarrow$ part then follows by a dual argument.
\end{proof}

It follows that $\sqsubseteq$ is invariant under products on either side, i.e. for all $l,m,n\in N$,
\begin{equation}
    \label{eq:resinv}
m\sqsubseteq n\qquad\Rightarrow\qquad lm\sqsubseteq ln\quad\text{and}\quad ml\sqsubseteq nl.
\end{equation}

\begin{prop}
\label{prop:partial}
    The restriction relation is a closed partial order relation on $N$.
\end{prop}

\begin{proof}
    For every $n\in N$, note $b_k=p_k(n^*n)\in B^1_+$ satisfies $n=\lim_knb_k$ so $n\sqsubseteq n$, showing that $\sqsubseteq$ is reflexive.  To see that $\sqsubseteq$ is also antisymmetric, say $m\sqsubseteq n\sqsubseteq m$ so we have $(a_k),(b_k)\subseteq B^1_+$ with $m=\lim_ka_km=\lim a_kn$ and $n=\lim_knb_k=\lim_kmb_k$ and hence
    \[m=\lim_ka_kn=\lim_ka_knb_k=\lim_kmb_k=n.\]
    To see that $\sqsubseteq$ is also transitive, say $l\sqsubseteq m\sqsubseteq n$, so we have $(a_k),(b_k)\subseteq B^1_+$ with $l=\lim_kla_k=\lim ma_k$ and $m=\lim_kmb_k=\lim_knb_k$.  Then we see that
    \[l=\lim_kma_k=\lim_kmb_ka_k=\lim_kma_kb_k=\lim_kla_kb_k\]
    and, likewise, $l=\lim_kma_k=\lim_knb_ka_k=\lim_kna_kb_k$, showing $(a_kb_k)$ witnesses $l\sqsubseteq n$.

    To see that $\sqsubseteq$ is closed, take $m_k\rightarrow m$ and $n_k\rightarrow n$ with $m_k\sqsubseteq n_k$, for all $k$. For every $\varepsilon>0$, this means that we have large enough $k$ with $\|m-m_k\|,\|n-n_k\|<\varepsilon$, and $b\in B^1_+$ with $\|m_k-m_kb\|,\|m_k-n_kb\|<\varepsilon$, hence $\|m-mb\|,\|m-nb\|<3\varepsilon$, showing $m\sqsubseteq n$.
\end{proof}

One last property of restriction that will be needed in \cref{BG0} is the following.

\begin{lemma}\label{lem:RR}
    For any $m,n\in N$,
    \begin{equation}
        m\sqsubseteq n\qquad\Rightarrow\qquad n-m\sqsubseteq n.
    \end{equation}
\end{lemma}

\begin{proof}
    If $m\sqsubseteq n$ then we have $(b_k)_{k\in\mathbb{N}}\subseteq B$ with $m=\lim_kb_km=\lim_kb_kn$.  By \cref{prop:approxunit}, we also have a net $(a_\lambda)_{\lambda\in\Lambda}\subseteq B$ with $n=\lim_\lambda a_\lambda n$ and $m=\lim_\lambda a_\lambda m$.  We can then turn these into limits of nets indexed by $\Lambda'=\Lambda\times\mathbb{N}$ in the product ordering, i.e. $m=\lim_{(\lambda,k)}b_km=\lim_{(\lambda,k)}b_kn$, $n=\lim_{(\lambda,k)}a_\lambda n$ and $m=\lim_{(\lambda,k)}a_\lambda m$ and hence
    \[n-m=\lim_{(\lambda,k)}(a_\lambda-b_k)(n-m)=\lim_{(\lambda,k)}(a_\lambda-b_k)n\in\mathrm{cl}(Bn)\subseteq N,\]
    thus showing that $n-m\sqsubseteq n$.
\end{proof}

\section{The Expectation}

Here we examine some additional properties of our expectation $E$.  First we prove that the $n^*$ in \eqref{StarStability} from \cref{def:semiCartan}\eqref{item:faithfulbistable} can actually be replaced by any $m\in N$ with $mn\in N$, thus showing that $E$ is \emph{bistable} in the sense of \cite[Definition 1.4]{Bice2023}.

\begin{prop}\label{Prop:Bistable}
    For all $m,n\in N$,
    \begin{equation}\label{Bistable}
        \tag{Bistable}mn\in B\qquad\Rightarrow\qquad E(m)n\in B.
    \end{equation}
\end{prop}

\begin{proof}
    First note that, for any $m\in N$ and $k\in\mathbb{N}$,
    \begin{equation}\label{pkE}
        E(p_k(mm^*)m)=p_k(E(m)m^*)m.
    \end{equation}
    To see this, it suffices to show that $E((mm^*)^km)=E(m)(m^*m)^k=(E(m)m^*)^km$, for all $k\geq1$.  The first equality follows from $m^*m\in B$, and the second equality is immediate when $k=1$.  If we assume that $E(m)(m^*m)^k=(E(m)m^*)^km$ holds then, as $E(m)m^*\in B$,
    \begin{align*}
        (E(m)m^*)^{k+1}m&=E(m)m^*(E(m)m^*)^km=E(m)m^*E(m)(m^*m)^k\\
        &=E(E(m)m^*m)(m^*m)^k=E(E(m))m^*m(m^*m)^k=E(m)(m^*m)^{k+1}.
    \end{align*}
    Thus the desired equality follows by induction.

    Now, for any $m,n\in N$ with $mn\in B$, just note that
    \[E(m)n=\lim_kE(p_k(mm^*)m)n=\lim_k p_k(E(m)m^*)mn\in\mathrm{cl}(BB)\subseteq B.\qedhere\]
\end{proof}

Bistability allows us to show that $E$ is deflationary with respect to the restriction relation on $N$.  In particular, we will need the property $E(n) \sqsubseteq n$ to characterise the units of our ultrafilter groupoid in \cref{lem:units and Un}.  In \cref{cor:masa}, we establish that $E$ being deflationary for \emph{all} normalisers implies that $B$ is a Cartan subalgebra.  A special case of the first part of \cref{prop:resE} below is proved in \cite[Theorem 3.1(4)]{ABCCLMR2023}, which in hindsight can be viewed as a particular instance of $E(n) \sqsubseteq n$.

\begin{prop}\label{prop:resE}
    For all $n\in N$, $E(n) \sqsubseteq n$.  In fact, $E(n)$ is the maximum of the poset $(n^\sqsupseteq\cap B,\sqsubseteq)$, i.e.
    \begin{equation}\label{Emax}
        E(n)=\max\{b\in B:b\sqsubseteq n\}.
    \end{equation}
\end{prop}

\begin{proof}
    To see that $E(n)\sqsubseteq n$, we use the sequence $(p_k)$ fixed in \cref{def:p_k poly}.  Specifically, for each $k\in\mathbb{N}$, note that
    \[b_k:=p_k(n^*E(n))=p_k(E(n)^*E(n))\in B^1_+,\]
    by the bistability of $E$ on $N$.  Next note that, using \eqref{pkE},
    \[E(n)=\lim_kE(np_k(n^*n))=\lim_knp_k(n^*E(n))=\lim_knb_k.\]
    It then further follows that
    \[E(n)=E(E(n))=\lim_kE(nb_k)=\lim_kE(n)b_k,\]
    showing that $E(n)\sqsubseteq n$.  Finally just note that if $b\in B$ and $b\sqsubseteq n$ then $b=E(b)\sqsubseteq E(n)$ so $E(n)$ is indeed the maximum $b\in B$ such that $b\sqsubseteq n$.
\end{proof}

By \eqref{Emax}, $E$ is uniquely determined on $N$. As $E$ is a contraction, it follows that the stable expectation $E$ is uniquely determined on $\mathrm{cl}(\mathrm{span}(N))= A$.

\begin{cor}\label{UniqueStableExpectation}
    The expectation $E$ in \cref{def:semiCartan} is unique.
\end{cor}

It also follows that our expectation $E$ is \emph{normal} with respect to $N$ as in the following.

\begin{prop} \label{prop:Normal}
    For any $a\in A$ and $n\in N$,
    \[\tag{Normal}\label{Normal}E(n^*an)=n^*E(a)n.\]
\end{prop}

\begin{proof}
    We show $E(n^*an) \sqsubseteq n^*E(a)n$ and $n^*E(a)n \sqsubseteq E(n^*an)$ which suffice by \cref{prop:partial}. As $N$ has dense span, it suffices to consider $a\in N$.  In this case, $E(a)\sqsubseteq a$ by \cref{prop:resE} and hence $n^*E(a)n\sqsubseteq n^*an$ by \eqref{eq:resinv}.  But $N \subseteq N(B)$, by \cref{lem:Normal}, so $n^*E(a)n\in B$ and hence $n^*E(a)n=E(n^*E(a)n)\sqsubseteq E(n^*an)$.
    
    For the reverse inequality, note that replacing $n$ with $n^*$ and $a$ with $n^*an$ yields
    \[nE(n^*an)n^*\sqsubseteq E(nn^*ann^*)=nn^*E(a)nn^*\]
    and hence for the $q_k$ defined after \cref{def:p_k poly}, $q_k(n^*n)n^*nE(n^*an)n^*nq_k(n^*n)\sqsubseteq q_k(n^*n)n^*nn^*E(a)nn^*nq_k(n^*n)$ so
    \[E(n^*an)=\lim_kp_k(n^*n)E(n^*an)p_k(n^*n)\sqsubseteq\lim_kp_k(n^*n)n^*E(a)np_k(n^*n)=n^*E(a)n.\qedhere\]
\end{proof}

Just as $B\subseteq N(B)$ and \cref{def:semiCartan}\eqref{item:faithfulbistable} imply the stronger *-free statements in \cref{cor:binormal} and \cref{Prop:Bistable} respectively, \cref{prop:Normal} also has the following *-free analog.

\begin{cor} \label{prop:Shiftable} \label{NormalImpliesShiftable}
For any $a\in A$ and $n\in N$,
\begin{equation}
    \tag{Shiftable}E(na)n=nE(an).
\end{equation}
\end{cor}

\begin{proof}
For any $a\in A$ and $n\in N$, \cref{prop:Normal} yields
\[E(nn^*na)n=nn^*E(na)n=nE(n^*nan)=nn^*nE(an).\]
It follows that $E(na)n=\lim_kE(p_k(nn^*)na)n=\lim_kp_k(nn^*)nE(an)=nE(an)$.
\end{proof}

It follows that $(A,N,B,E)$ is a well-structured semimodule and, in particular, $(N,B,E)$ is a well-structured semigroup, in the sense of \cite[Definition 1.5]{Bice2023}.  Thus we are again free to use any of the results for well-structured semimodules and semigroups appearing in \cite{Bice2023}.

Lastly, we exhibit one more property of the expectation that will be used to prove the transitivity of relation $\sim_U$ considered in Section~\ref{sec:sigma}.

\begin{lemma}\label{lem:homom-like}
    For any character $\phi$ on $B$ and any $m,n\in N$,
    \[\phi(E(m))\neq0\qquad\Rightarrow\qquad\phi(E(mn))=\phi(E(m)E(n)).\]
\end{lemma}

\begin{proof}
    As $E(m)\sqsubseteq m$ 
    by \cref{prop:resE}, we have $(b_k)\subseteq B$ such that  \[E(m)=\lim_kb_kE(m)=\lim_kb_km.\]
    Then \[\phi(E(m))=\lim_k\phi(b_kE(m))=\lim_k\phi(b_k)\phi(E(m))\] so $\lim_k\phi(b_k)=1$, as $\phi(E(m))\neq0$.  Thus,
    \[\hspace{-3pt}\phi(E(mn))=\lim_k\phi(b_k)\phi(E(mn))=\lim_k\phi(E(b_kmn))=\phi(E(E(m)n))=\phi(E(m)E(n)).\qedhere\]
\end{proof}

We will end this section with an application of \cref{Prop:Bistable} which will allow us to enlarge our given Cartan semigroup $N$ to include all compatible sums without changing the associated semi-Cartan algebra, and will be used in \cref{sec:rep} to describe our main results.  Indeed, let us denote the compatible sums in $N$ and its closure by
\begin{align}\label{SigmaN}
\mathrm{csum}(N)&=\{\sum_{j=1}^kn_j:n_1,\ldots,n_k\in N\text{ are compatible}\}, \text{ and}\\
    \label{SigmabarN}\overline{\mathrm{csum}}(N)&=\mathrm{cl}(\mathrm{csum}(N)).
\end{align}
\begin{prop}The closure
    $\overline{\mathrm{csum}}(N)$ is a Cartan semigroup with $B=C^*(\overline{\mathrm{csum}}(N)_+)$.
\end{prop}
\begin{proof}
    We immediately see that $N\subseteq\mathrm{csum}(N)\subseteq\overline{\mathrm{csum}}(N)$ so $N_+\subseteq\overline{\mathrm{csum}}(N)_+$, while the definition of compatibility yields $\mathrm{csum}(N)_+\subseteq B$ and hence $\overline{\mathrm{csum}}(N)_+\subseteq B_+$, from which it follows that $C^*(\overline{\mathrm{csum}}(N)_+)=C^*(N_+)=B$.  
    To see that $\mathrm{csum}(N)$ is a semigroup, take compatible sums $m=\sum_{j=1}^km_j$ and $n=\sum_{j=1}^ln_j$ so $mn=\sum_{i=1}^k\sum_{j=1}^lm_in_j$.  For all $i,i'\leq k$ and $j,j'\leq l$, we then see that $n_j^*m_i^*m_{i'}n_{j'}\in n_j^*Bn_{j'}\subseteq B$, by \cref{cor:binormal}~\eqref{Binormal}, and likewise, $n_jm_im_{i'}^*n_{j'}^*\in B$.  This shows that $mn$ is also a compatible sum of elements in $N$, i.e. $mn\in\mathrm{csum}(N)$, showing that $\mathrm{csum}(N)$ and hence $\overline{\mathrm{csum}}(N)$ is a semigroup.  As $N\subseteq\overline{\mathrm{csum}}(N)=\overline{\mathrm{csum}}(N)^*$, to see that $\overline{\mathrm{csum}}(N)$ is a Cartan semigroup, it only remains to verify stability.  But if $n_1,\ldots,n_k\in N$ are compatible then
    \[E(\sum_{j=1}^kn_j)(\sum_{j=1}^kn_j)^*=\sum_{i,j=1}^kE(n_i)n_j^*\in B,\]
    by \eqref{Bistable} in \cref{Prop:Bistable}, showing that $E$ satisfies \cref{def:semiCartan}\eqref{item:faithfulbistable}, that is, $E$ is stable on $\mathrm{csum}(N)$ and hence on $\overline{\mathrm{csum}}(N)$.
\end{proof}
Iterating compatible sums then yields a larger summable Cartan semigroup with the same semi-Cartan subalgebra, i.e. setting $\overline{\mathrm{csum}}^1(N)=\overline{\mathrm{csum}}(N)$ and $\overline{\mathrm{csum}}^{n+1}(N)=\overline{\mathrm{csum}}^n(N)$ it follows that $\mathrm{cl}(\bigcup_{n\in\mathbb{N}}\overline{\mathrm{csum}}^n(N))$ is the enlarged summable Cartan semigroup that we are after.  Later we will see that there is actually no need to iterate, as $\overline{\mathrm{csum}}(N)$ is already a summable Cartan semigroup isomorphic to one of the form $\mathrm{cl}(\Nc(\Sigma;G))$ like in \cref{TwistedGroupoid->CartanSemigroup}.

\section{The Domination Relation}\label{sec:domination}

Following \cite[Definition~5.1]{Bice2022}, we define the \emph{domination} relation $<$ by
\begin{align*}
m<_sn\qquad&\Leftrightarrow\qquad m,s,n\in N,\quad sm,ms,sn,ns\in B\quad\text{and}\quad nsm=m=msn.\\
m<n\qquad&\Leftrightarrow\qquad\exists s\in N\ \text{ such that } m<_sn.
\end{align*}

Domination is the second but more important transitive relation we study on $N$, given that the groupoid in \cref{sectionG} onwards consists of ultrafilters with respect to the domination relation.  As with $\sqsubseteq$, we can characterise $<$ on monomial semigroups of twisted groupoid C*-algebras in terms of supports.  Indeed, for any topological space $X$, let us define the compact containment relation $\Subset$ on subsets of $X$ by
\[Y\Subset Z\qquad\Leftrightarrow\qquad\exists\text{ compact }K\ \text{ such that }Y\subseteq K\subseteq Z.\]

\begin{prop}\label{DominationSupports}
    If $A$ is a $D$-contractive C*-completion of $C_c(\Sigma;G)$, for some twist $q:\Sigma\twoheadrightarrow G$, and $N=\mathrm{cl}(\Nc(\Sigma;G))$ $($so $B=\mathrm{cl}(\Bc(\Sigma;G)))$ then, for all $m,n\in N$,
    \[m<n\qquad\Leftrightarrow\qquad\supp^\circ(j(m))\Subset\supp^\circ(j(n)).\]
\end{prop}

\begin{proof}
   By \cref{jSemigroupIso}, the $j$-map identifies $N$ with $N_0(\Sigma;G)$. By \cite[Theorem 3.3]{BC2021}, $N$ is a bumpy semigroup, in fact, $N$ is compact bumpy.   Thus our claim follows from \cite[Proposition 4.4]{BC2021}.  
\end{proof}

While the multiple conditions defining domination may seem a little intimidating at first, in practice it suffices to verify only some of them, e.g. for all $m,n,s\in N$,
\begin{equation*}
    ms,sn,ns\in B\quad\text{and}\quad m=msn\qquad\Rightarrow\qquad m<_sn,
\end{equation*}
by \cite[Proposition 5.4]{Bice2022}.  Likewise, it is enough to verify $sm,ns,sn\in B$ and $nsm=m$.  If we only care about $<$ then we can simplify further, e.g. for all $m,n,s\in N$,
\begin{equation}\label{OneSidedDomination}
    ms,sn\in B\quad\text{and}\quad m=msn\qquad\Rightarrow\qquad m<n,
\end{equation}
thanks to \cref{lem:Symmetry} and \cite[Proposition 8.6]{Bice2022}.  Even $sn\in B$ here is unnecessary, which means that $<$ also agrees with the domination relation considered in \cite{BC2021}. We continue proving several necessary properties of $<$.

\begin{lemma}\label{OpenOneSidedDomination}
For any $l,s,m\in N$, there exists an $\varepsilon>0$ such that
\[ls\in B, l=lsm \text{ and } \|m-n\|<\varepsilon \quad \Rightarrow \quad l<n.\]
\end{lemma}

\begin{proof}
Take any $\varepsilon>0$ with $\varepsilon\|s\|^2(2\|m\|+\varepsilon)<1$.  If $\|m-n\|<\varepsilon$ then
\begin{align*}
\|E(m^*s^*)sm-E(n^*s^*)sn\|&\leq\|E(m^*s^*)sm-E(m^*s^*)sn\|+\|E(m^*s^*)sn-E(n^*s^*)sn\|\\
&\leq\|m-n\|\|s\|^2(\|m\|+\|n\|)\\
&\leq\|m-n\|\|s\|^2(2\|m\|+\|m-n\|)\\
&<1.
\end{align*}
If $l=lsm$ then $l=lm^*s^*$, by \cref{mnstar}, and hence
\[l^*l=l^*lsm=E(l^*l)sm=E(l^*lm^*s^*)sm=l^*lE(m^*s^*)sm.\]
Identifying $B$ with $C_0(X_B)$, this means that $E(m^*s^*)sm$ is $1$ on the support of $l^*l$ and hence $E(n^*s^*)sn$ takes values greater than some $\delta>0$ on the support of $l^*l$.  Letting $f$ be a function on $\mathbb{R}_+$ with $f(0)=0$ and $f(r)=1/r$, for all $r>\delta$, it follows that $l^*l=l^*lf(E(n^*s^*)sn)E(n^*s^*)sn$.  Letting $t=f(E(n^*s^*)sn)E(n^*s^*)s$, it follows that $l=ltn$.  Also $lt\in lBs\subseteq B$, by \cref{cor:binormal}, and $tn\in BE(n^*s^*)sn\subseteq BB\subseteq B$, by \cref{def:semiCartan}\eqref{item:faithfulbistable}~\eqref{StarStability}.  This shows that $l<n$, by \eqref{OneSidedDomination}.
\end{proof}

For any $T\subseteq N$, we let
\[T^<=\{s\in N: \: \text{there exists } t\in T\ \text{ such that }t<s\},\] 
i.e.~the up-set of $T$ with respect to $<$. Similarly, $T^>$ denotes the down-set of $T$. When $T=\{t\}$ is a singleton, we usually omit the curly braces and just write $t^<$ (respectively $t^>$) instead of $\{t\}^<$ (respectively $\{t\}^>$). 

The principal up-set $n^<$ defined by any $n\in N$ is an open subset of $N$, thanks to \cref{OpenOneSidedDomination}.  Taking $n=m$ in \cref{OpenOneSidedDomination}, we also obtain the even weaker characterisation of $<$ alluded to earlier, i.e. for all $l,s,m\in N$,
\begin{equation}\label{WeakDomination}
    ls\in B\quad\text{and}\quad l=lsm\qquad\Rightarrow\qquad l<m.
\end{equation}
In contrast, if we fix some $s\in N$ then $n^{<_s}$ is a closed subset of $N$.  In fact, $<_s$ is a closed relation, i.e.~a closed subset of $N\times N$, as we immediately see from the definition of $<_s$ and the fact that $B$ and $N$ are closed.

There are a few more observations about $<$ that we can make at this point.  First note that any element of $N$ dominated by an element of $B$ must in fact also lie in $B$, i.e.
\begin{equation}\label{Bbounded}
    n \in N, b \in B, \text{ and } n<b \qquad \Rightarrow \qquad n \in B.
\end{equation}
Indeed, if $n<_sb\in B$ then $n=nsb\in BB\subseteq B$. Further, if $m,n,s \in N$ and $b \in B$,
\begin{equation} \label{eqn:mult by B and <}
    m<_s n \qquad \Rightarrow \qquad mb<_s n.
    \end{equation}

Also recall that any property of the domination relation on well-structured semigroups must apply, in particular, to Cartan semigroups.  For example, this is true for the results in \cite[\S5]{Bice2023} relating domination to the expectation, e.g. showing that $<$ is $E$-invariant, i.e.
\begin{equation}\label{DominationExpectation}
    m<_sn\qquad\Rightarrow\qquad E(m)<_{E(s)}E(n).
\end{equation}
Likewise, we immediately see from the definition that $<$ is *-invariant, i.e.
    \begin{equation}\label{eq:stardom}
    m<_sn\qquad\Rightarrow\qquad m^*<_{s^*}n^*.\end{equation}
Alternatively, we can leave $m$ fixed and instead switch the last two terms, i.e.
\begin{equation}\label{StarSwitch}
    m<_sn\qquad\Rightarrow\qquad m<_{n^*}s^*.
\end{equation}
Indeed if $m<_sn$ then $mn^*=msnn^*\in BB\subseteq B$ and $m=mn^*s^*$ by Lemma~\ref{mnstar}.

Next we note that domination is auxiliary to restriction in the following sense.

\begin{lemma}\label{prop:aux}
    For all $k,l,m,n\in N$,
    \[\tag{Auxiliarity}\label{Auxiliarity}k\sqsubseteq l<m\sqsubseteq n\qquad\Rightarrow\qquad k<n.\]
\end{lemma}

\begin{proof}
    If $k\sqsubseteq l$ and $l<_sm$ then we have $(b_j)\subseteq B$ with $k=\lim_jlb_j$ so $lb_j<_sm$, for all $j$, by \cref{eqn:mult by B and <}, and hence $k<_sm$.  On the other hand, if $l<_sm\sqsubseteq n$ then we claim that $l<_{sms}n$.  To see this, take $(b'_j)\subseteq B$ with $m=\lim_j b'_j m=\lim_j b'_j n$ and note that
    \[lsmsn=\lim_jlsb'_jmsn=\lim_jlsmsb'_jn=lsmsm=l.\]
    The other conditions needed to prove $l<_{sms}n$ are verified in the same way. 
\end{proof}

We also note that any $n\in N$ can be approximated by the elements it dominates.  In particular this shows that, while $<$ itself is not reflexive, its closure is.

\begin{lemma}\label{DominatedApproximation}
For every $n\in N$, there is a sequence $(n_k)\subseteq n^>$ with $n_k\rightarrow n$.
\end{lemma}

\begin{proof}
Pick continuous functions $f_k$ and $g_k$ on $\mathbb{R}_+$ such that $f_k(0)=g_k(0)=0$, $f_k(x)g_k(x)x=f_k(x)$ for all $x\in\mathbb{R}_+$, and $f_k(x)=1$ for all $x\in[1/k,\infty)$.  Note this last condition implies that $f_k(n^*n)$ is a right approximate unit for $n$ and hence $n_k\rightarrow n$, where $n_k=nf_k(n^*n)$.  Letting $s_k=g_k(n^*n)n^*$, we see that
\begin{align*}
    n_ks_k&=nf_k(n^*n)g_k(n^*n)n^*\in nBn^*\subseteq B, 
\end{align*}    
as $N\subseteq N(B)$ by \cref{lem:Normal}, and
\begin{align*}
    n_ks_kn&=nf_k(n^*n)g_k(n^*n)n^*n=nf_k(n^*n)=n_k.
\end{align*}
Thus $n_k<_{s_k} n$, for all $k\in\mathbb{N}$, i.e. $(n_k)\subseteq n^>$.
\end{proof}

Another key property of domination is the following interpolation result.

\begin{lemma}\label{Interpolation}
    If $m<_sn$, then there exists $l\in nB_+\cap B_+n$ such that $m<_sl<n$.
\end{lemma}

\begin{proof}
     Take continuous functions $g$ and $h$ on $\mathbb{R}_+$ with $g(0)=h(0)=0$, $g(1)=1$ and $g(x)h(x)x=g(x)\leq x^{-1}$, for all $x>0$.  If $m<_sn$ then let $l=ng(sn)=g(ns)n$ and note that $msl=msng(sn)=mg(sn)=m$, as $g(1)=1$.  Likewise, $lsm=m$ and $ls,sl\in B$ and hence $m<_sl$.  On the other hand, letting $t=h(sn)s$, we see that $tl=h(sn)sng(sn)=g(sn)\in B$ and $ntl=ng(sn)=l$ and hence $l<n$.
\end{proof}

\section{The Ultrafilter Groupoid}\label{sectionG}

Recall that a \emph{filter} (w.r.t. $<$) is a subset $F\subseteq N$ such that, for all $m,n\in N$,
\[\tag{Filter}m,n\in F\qquad\Leftrightarrow\qquad \text{ there exists } l\in F\text{ such that }l<m,n.\]
Note $\Rightarrow$ is saying that $F$ is \emph{down-directed} while $\Leftarrow$ is saying that $F$ is \emph{up-closed}, i.e. $F^<\subseteq F$.  As $0^<=N$, a filter $F$ is a proper subset of $N$ precisely when $0\notin F$.  The \emph{ultrafilters} are the maximal proper filters, which we denote by
\[G=G_N:=\{U\subseteq N:U\text{ is an ultrafilter}\}.\]
First we show that ultrafilters always exist.

\begin{prop}\label{UltrafiltersExist}
    Every $n\in N\setminus\{0\}$ is contained in an ultrafilter $U\subseteq N$.
\end{prop}

\begin{proof}
    Given any $n\in N\setminus\{0\}$, \cref{DominatedApproximation} yields $m\in n^>\setminus\{0\}$.  Then \cref{Interpolation} yields a sequence $(m_k)\subseteq N$ such that $m<m_{k+1}<m_k<n$, for all $k\in\mathbb{N}$.  We then obtain a filter $F=\bigcup_{k\in N}m_k^<$ containing $n$ but not $0$ and hence the Kuratowski-Zorn lemma then yields an ultrafilter $U$ such that $n\in F\subseteq U$.
\end{proof}

Like maximal ideals, ultrafilters carry a natural topology generated by $\{\mathcal{U}_n\}_{n\in N}$, where
\[\mathcal{U}_n=\{U\in G:n\in U\}.\]
(As each $U\in G$ is a filter, we immediately see that $\{\mathcal{U}_n\}_{n\in N}$ forms a basis of $G$.)  In contrast to maximal ideals, however, the ultrafilters also carry a natural groupoid structure.

\begin{prop}\label{UltraGroupoid}
    The ultrafilters $G$ form an \'etale groupoid under the product
    $T\cdot U=(TU)^<$ defined if and only if $0\notin TU$ and the inverse of each $U\in G$ is given by $U^{-1}=U^*:=\{u^*:u\in U\}$, with $\mathsf{s}(U):=(U^{-1}U)^<$ and $\mathsf{r}(U):=(UU^{-1})^<$.
\end{prop}

\begin{proof}
    By \cite[Theorems 8.1, 8.4 and 12.8]{Bice2022}, $G$ is an \'etale groupoid under the product $T\cdot U=(TU)^<$, where the inverse of each $U\in G$ is given by 
    \begin{equation}\label{eq:inv}U^{-1}=\{s\in N: \: \text{ there exists } m, n \in U \text{ such that } m<_s n\}.
    \end{equation} 
    But by \eqref{StarSwitch}, $m<_sn$ is equivalent to $m<_{n^*}s^*$ so $s\in U^{-1}$ is equivalent to $s^*\in U^<=U$, thus showing that $U^{-1}=U^*$.  As for when $T\cdot U$ is defined, this happens precisely when $\mathsf{s}(T)=\mathsf{r}(U)$ -- it only remains to show that this is equivalent to $0\notin TU$.  To see this, note first that if $\mathsf{s}(T)=\mathsf{r}(U)$ then $\mathsf{s}(T)=\mathsf{s}(T)\cdot\mathsf{r}(U)=(T^*TUU^*)^<$.  Then $0\in TU$ would imply $\mathsf{s}(T)=0^<=N$, a contradiction, so $\mathsf{s}(T)=\mathsf{r}(U)$ implies $0\notin TU$.  Conversely, if $\mathsf{s}(T)\neq\mathsf{r}(U)$ then $(\mathsf{s}(T)\mathsf{r}(U))^<$ is filter containing both $U$ and $T$, as $U\subseteq(t^*tU)^<$ and $T\subseteq(Tuu^*)^<$, for any $t\in T$ and $u\in U$.  The maximality of $T$ and $U$ then implies that $0\in(\mathsf{s}(T)\mathsf{r}(U))^<$ and hence $0\in T(\mathsf{s}(T)\mathsf{r}(U))^<U\subseteq TU$.
\end{proof}

\begin{remark}\label{UltrafiltersArePoints}
In the motivating situation of \cref{TwistedGroupoid->CartanSemigroup}, the ultrafilter groupoid $G_N$ corresponds with the groupoid $G$ that we started with.  Indeed, if $A$ is a $D$-contractive C*-completion of $C_c(\Sigma;G)$, for some twist $q:\Sigma\twoheadrightarrow G$, and $N=\mathrm{cl}(\Nc(\Sigma;G))$ is its monomial semigroup then
\[g\mapsto U_g:=\{n\in N:j(n)(q^{-1}\{g\})\neq\{0\}\}\]
is an \'etale groupoid isomorphism from the original groupoid $G$ onto the ultrafilter groupoid $G_N$. This follows by essentially the same argument as in \cite[Theorem 5.3]{BC2021} applied to the bumpy semigroup $N$, see also \cref{DominationSupports}.
\end{remark}

Next note that, by \cite[Proposition 8.2]{Bice2022}, the units of $G$ are precisely the ultrafilters that have nonempty intersection with $B$, that is
\begin{equation}\label{G0B}
    G^{(0)}=\bigcup_{b\in B}\mathcal{U}_b.
\end{equation}
Also, the inverses of ultrafilters in $\mathcal{U}_n$ are just the ultrafilters containing $n^*$, i.e.
\[\mathcal{U}_n^{-1}=\mathcal{U}_{n^*}.\]
In particular, $\mathcal{U}_n^{-1}\mathcal{U}_n=\mathcal{U}_{n^*}\mathcal{U}_n\subseteq\mathcal{U}_{n^*n}\subseteq G^{(0)}$ and, likewise, $\mathcal{U}_n\mathcal{U}_n^{-1}\subseteq G^{(0)}$, showing that each $\mathcal{U}_n$ is a bisection of $G$.  Moreover, by \cref{lem:Symmetry} and \cite[Theorems 8.11 and 12.8]{Bice2022},
\begin{equation}\label{eq:unm}
\mathcal{U}_{mn}=\mathcal{U}_m\mathcal{U}_n,
    \end{equation}
for all $m,n\in N$.  So we see that $\{\mathcal{U}_n\}_{n\in N}$ is a basis of bisections of $G$ forming an inverse semigroup under pointwise products and inverses.

Using the expectation, we also can strengthen \eqref{G0B} like in \cite[Proposition 6.4]{Bice2023}.

\begin{prop}\label{lem:units and Un}
    For all $n\in N$,
    \[\VV_{E(n)}= \VV_n\cap G^{(0)}.\]
\end{prop}

\begin{proof}
    If $U\in \VV_{E(n)}$, then there exists $m\in U$ with $m <E(n)$ and then $m<E(n)\sqsubseteq n$ by \cref{prop:resE}. So $m<n$ by \eqref{Auxiliarity} in \cref{prop:aux}, and hence $U\in \VV_m\subseteq \VV_n$.  Certainly $\VV_{E(n)}\subseteq G^{(0)}$ because $E(n)\in B$ and the unit space is given by \eqref{G0B}.  So $\VV_{E(n)}\subseteq \VV_n\cap G^{(0)}$.

    Conversely, if $U\in \VV_n\cap G^{(0)}$ then, taking any $b\in U\cap B$, we must further have some $m\in U$ with $m<b,n$ and hence $m=E(m)< E(n)$ by \eqref{Bbounded} and \eqref{DominationExpectation}, so $U\in \VV_m\subseteq \VV_{E(n)}$.
\end{proof}

In contrast to \eqref{G0B}, we also have the following.

\begin{prop}\label{prop:G0complement}
    The non-unit ultrafilters are precisely those intersecting $E^{-1}\{0\}$, i.e.
    \begin{equation*}
        G\setminus G^{(0)}=\bigcup_{E(n)=0}\VV_n.
    \end{equation*}
\end{prop}

\begin{proof}
    By \cref{lem:units and Un}, $E(n)=0$ implies $\VV_n\cap G^{(0)}= \VV_{E(n)}= \VV_0= \emptyset$, thus showing that $\bigcup_{E(n)=0}\VV_n\subseteq G\setminus G^{(0)}$.  Conversely, if $U\in G\setminus G^{(0)}$ then, in particular, $U$ and hence $E(U)$ is directed, by \eqref{DominationExpectation}.  Then $E(U)^<$ is a filter containing $U$ (as $m<n$ implies $E(m)<n$, by \cref{prop:aux}~\eqref{Auxiliarity}, because $E(m)\sqsubseteq m<n$).  If $0$ were not in $E(U)$ then this would imply $E(U)^<=U$, by the maximality of $U$.  But then $E(U)\subseteq U$ and hence $U\in G^{(0)}$, a contradiction.  Thus $0\in E(U)$ so $U\in\bigcup_{E(n)=0}\VV_n$, showing $G\setminus G^{(0)}\subseteq\bigcup_{E(n)=0}\VV_n$.
\end{proof}

\begin{cor} \label{prop:G is Hausdorff}
    The ultrafilter groupoid $G$ is Hausdorff.
\end{cor}

\begin{proof}
    For any $T,U\in G^{(0)}$, if $T\neq U$ then $T\cdot U$ is not defined and hence $0\in TU$.  Taking any $m\in T$ and $n\in U$ with $mn=0$, it follows that $\VV_m$ and $\VV_n$ are disjoint neighbourhoods of $T$ and $U$ respectively, showing that $G^{(0)}$ is Hausdorff.  Now \cref{prop:G0complement} implies $G\setminus G^{(0)}$ is open so $G^{(0)}$ is closed, from which it follows that $G$ is Hausdorff.
\end{proof}

One fact worth noting is that the restriction relation $\sqsubseteq$ is stronger than the corresponding inclusion relation for basic subsets of ultrafilters, that is for all $m,n\in N$,
\begin{equation}\label{RestrictionImpliesInclusion}
    m\sqsubseteq n\qquad\Rightarrow\qquad\mathcal{U}_m\subseteq\mathcal{U}_n.
\end{equation}
Indeed, if $U\in\mathcal{U}_m$ then $m\in U$ so we have some $l\in U$ with $l<m\sqsubseteq n$ and hence $l<n$, by \eqref{Auxiliarity} in \cref{prop:aux}, which then implies $n\in U$ and hence $U\in\mathcal{U}_n$.
Together with the above results, this observation yields another corollary.

\begin{cor}\label{BG0}
    For all $n\in N$, the bisection $\mathcal{U}_n\subseteq G^{(0)}$ if and only if $n\in B$.
\end{cor}

\begin{proof}
    If $n\in B$ then $\mathcal{U}_n\subseteq G^{(0)}$ by \eqref{G0B}.  Conversely, if $n\in N\setminus B$ then $n\neq E(n)$ and hence $0\neq n-E(n)\sqsubseteq n$, by \cref{prop:resE} and \cref{lem:RR}.  Then \cref{UltrafiltersExist}, \cref{prop:G0complement} and \eqref{RestrictionImpliesInclusion} yield $\emptyset\neq\mathcal{U}_{n-E(n)}\subseteq\VV_n\setminus G^{(0)}$ and hence $\mathcal{U}_n\nsubseteq G^{(0)}$.
\end{proof}

Recall that a proper nonempty closed ideal $I$ in a commutative C*-algebra is maximal precisely when it is prime, meaning that $a\in I$ or $b\in I$ whenever $ab\in I$.  Likewise, ultrafilters are precisely the proper filters that are `additively prime'.

\begin{prop}\label{Ultra+}
A proper nonempty filter $U\subseteq N$ is an ultrafilter if and only if for all $m,n\in N$,
\begin{equation}
\label{eq:Ultra+}
m+n\in U\qquad\Rightarrow\qquad m\in U\quad\text{or}\quad n\in U.
\end{equation}
\end{prop}

\begin{proof}
Take a proper non-empty filter $U\subseteq N$.  If $U$ is not an ultrafilter, then $U$ is contained in a strictly larger proper filter $T$.  Take $t\in T\setminus U$ and $u\in U$.  As $T$ is a filter, we have $m,n\in T$ with $m<n<t,u$.  Take $n'$ with $m<_{n'}n$, necessarily with $n'\in T^{-1}=T^*$, and note that $u-un'n\notin T$ -- otherwise we would obtain the contradiction $0=(u-un'n)n'm\in TT^*T\subseteq T$.  Taking $t'$ with $n<_{t'}t$, we also note that $un'n<_{t'}t$.  Indeed, taking $u'$ with $n<_{u'}u$, we see that $un'nu',nu'ut'\in B$, by \cref{cor:binormal}, and hence $un'nt'=un'nu'ut'\in B$, by \cite[Proposition 1.6]{Bice2023}, as required.  It follows that $un'n\notin U$ -- otherwise $t\in U^<\subseteq U$, contradicting our choice of $t$.  Thus $(u-un'n)+un'n=u\in U$ even though $u-un'n\notin T\supseteq U$ and $un'n\notin U$, showing that $U$ fails to satisfy \eqref{eq:Ultra+}.

Conversely, looking for a contradiction, say we have an ultrafilter $U$ with $m+n\in U$, for some $m,n\in N\setminus U$.  By \cref{DominatedApproximation}, we have sequences $(m_k)\subseteq m^>$ and $(n_k)\subseteq n^>$ dominated by $m$ and $n$ respectively with $m_k\rightarrow m$ and $n_k\rightarrow n$, hence $m_k+n_k\rightarrow m+n$.  By \cref{OpenOneSidedDomination}, $U=\bigcup_{u\in U}u^<$ is open and hence $m_j+n_j\in U$, for some $j\in\mathbb{N}$.  By \cref{Interpolation}, we then have $(r_k),(s_k)\subseteq N$ with $m_j<r_{k+1}<r_k<m$ and $n_j<s_{k+1}<s_k<n$, for all $k\in\mathbb{N}$.

Now note that we have a filter $T$ containing $U$ and $m$ given by
\[T=\{t>E(r_ku^*)u: \text{ for some }u\in U\text{ and }k\in\mathbb{N}\},\]
thanks to \cite[Proposition 5.1 and Lemma 5.3]{Bice2023} and \cite[Proposition 5.7 and 5.8]{Bice2022}.  As $U$ is an ultrafilter and $m\in T\setminus U$, it follows that $T=N$ so $E(r_ku^*)u=0$ and hence $E(r_ku^*)=0$, for some $u\in U$ and $k\in\mathbb{N}$.  Taking $r\in N$ with $m_j<_rr_k$, we see that $E(m_ju^*)=E(m_jrr_ku^*)=m_jrE(r_ku^*)=0$ as well.  Likewise $E(n_jv^*)=0$, for some $v\in U$.  Taking $w\in U$ with $w<u,v$, it follows that $E(m_jw^*)=0=E(n_jw^*)$ and hence $0=E(w^*(m_j+n_j))\in E(U^*U)\subseteq \mathsf{s}(U)$, a contradiction.
\end{proof}

We can now show that maximal ideals in $B$ are just the complements of unit ultrafilters and that the resulting map is even a homeomorphism from $G^{(0)}$ onto the maximal ideal space $X_B$ (as described in \cref{NormedSpaces}).

\begin{thm}\label{ComplementHomeo}
    The map $h:G^{(0)}\rightarrow X_B$ given by $h(U)=B\setminus U$ is a homeomorphism.
\end{thm}

\begin{proof}
    If $U\in G^{(0)}$ then $U\cap B\neq\emptyset$ by \eqref{G0B}, so $B\setminus U$ is a proper subset of $B$.  Also $0\notin U$ so $0\in B\setminus U$ and, in particular, $B\setminus U$ is not empty.  As $U=\bigcup_{m\in U}m^<$ is open by \cref{OpenOneSidedDomination}, and $B$ is closed, so is $B\setminus U$.  By \cref{Ultra+}, $B\setminus U$ is closed under addition.  Also, for any $b\in B\setminus U$ and $c\in B$, we see that $bc\in B\setminus U$ as well -- if we had $bc\in U$ then we would have $m\in U$ with $m<bc$ and hence $m<b$, by \cref{lem:Symmetry} and \cite[Proposition 8.7]{Bice2022}, implying that $b\in U^<\subseteq U$, a contradiction.  Thus $B\setminus U$ is a proper nonempty closed ideal.  Moreover, for any $b,c\in B$ with $bc\in B\setminus U$, we must have either $b\in B\setminus U$ or $c\in B\setminus U$ -- otherwise $b,c\in U$ and hence $bc\in UU\subseteq U$, a contradiction.  Thus $B\setminus U$ is also prime and hence a maximal ideal, i.e. $h$ does indeed map $G^{(0)}$ to $X_B$.

    To see that $h$ is injective, take distinct $T, U\in G^{(0)}$ so we have some $t\in T\setminus U$.  Thus $T\in\mathcal{U}_t\cap G^{(0)}=\mathcal{U}_{E(t)}$ by \cref{lem:units and Un}, and hence $U\in G^{(0)}\setminus\mathcal{U}_t= G^{(0)}\setminus\mathcal{U}_{E(t)}$ by \cref{lem:units and Un} again. It follows that $E(t)\in T\setminus U$ and hence $E(t)\in(B\setminus U)\setminus(B\setminus T)$.  In particular, $B\setminus U\neq B\setminus T$, as required.

    To see $h$ is surjective, take $I\in X_B$.  For any $b,c\in B\setminus I$, note $b(I)\neq0\neq c(I)$ under the identification of $B$ with $C_0(X_B)$.  Taking $d\in B$ with $I\in\mathrm{supp^\circ}(d)\Subset\mathrm{supp^\circ}(b)\cap\mathrm{supp^\circ}(c)$, we see that $d<b,c$, by \cref{DominationSupports}.  So $B\setminus I$ is a proper (because $0\notin B\setminus I$) directed subset and Kuratowski-Zorn yields an ultrafilter $U$ containing $B\setminus I$.  So $B\setminus U\subseteq I$ and hence $B\setminus U=I$, as we have already shown that $B\setminus U$ is a maximal ideal.

    By \cref{lem:units and Un}, $\{\VV_b\}_{b\in B}$ is basis for the topology on $G^{(0)}$.  And for each $b\in B$,
    \[h(\VV_b)=\{h(U):U\in \VV_b\}=\{B\setminus U:b\in U\in G^{(0)}\}=\{I\in X_B: b\notin I\}=X_b.\]
    As $\{X_b\}_{b\in B}$ is a basis for $X_B$, this shows that $h$ is a homeomorphism.
\end{proof}

It follows that, for every $U\in G^{(0)}$, there is a unique $I\in X_B$ such that
\[U\cap B=\{b\in B:b(I)\neq0\}.\]
Indeed, by the above result we can take $I=B\setminus U\in X_B$ and then
\[\{b\in B:b(I)\neq0\}=\{b\in B:b\notin I\}=\{b\in B:b\in U\}=U\cap B.\]
Uniqueness is thus immediate from the injectivity of $h$ above.

We have already noted in \cref{DominationSupports} that domination corresponds to compact containment of supports in twisted groupoid C*-algebras.  This remains valid in general if we replace supports with the corresponding subsets of the ultrafilter groupoid.

\begin{prop}\label{prop:rainbow}
For all $m,n\in N$, $m< n$ if and only if $\mathcal{U}_m\Subset\mathcal{U}_n$.
\end{prop}

\begin{proof}
    By \cref{DominationSupports}, \eqref{G0B} and \cref{ComplementHomeo}, we have that \cref{prop:rainbow} holds for all $m,n\in B$.  For all $m,n\in N$, we also know from \eqref{eq:unm} that
    \[\mathcal{U}_m\mathcal{U}_n=\mathcal{U}_{mn}.\]
    Now if $m<n$, then $\mathcal{U}_m\subseteq\mathcal{U}_n$ because ultrafilters are upwards closed.  To show $\mathcal{U}_m\Subset\mathcal{U}_n$ it suffices to show that $\mathsf{s}(\mathcal{U}_m)\Subset \mathsf{s}(\mathcal{U}_n)$, as the source map $\mathsf{s}$ is a homeomorphism on any open bisection.  To see this, note that $m<n$ implies $m^*<n^*$, by \eqref{eq:stardom}, so $m^*m<n^*n$, by \cite[Proposition 5.7]{Bice2022}, and hence
    \[\mathsf{s}(\mathcal{U}_m)=\mathcal{U}_m^{-1}\mathcal{U}_m=\mathcal{U}_{m^*}\mathcal{U}_m=\mathcal{U}_{m^*m}\Subset\mathcal{U}_{n^*n}=\mathsf{s}(\mathcal{U}_n).\]

    Conversely, if $\mathcal{U}_m\Subset \mathcal{U}_n$ then, in particular, $\mathcal{U}_m\subseteq \mathcal{U}_n$ and hence
    \[\mathcal{U}_{mn^*}=\mathcal{U}_m\mathcal{U}_{n^*}=\mathcal{U}_m\mathcal{U}_n^{-1}\subseteq \mathcal{U}_n\mathcal{U}_n^{-1}\subseteq G^{(0)},\]
    which implies $mn^*\in B$ by Corollary~\ref{BG0}.  Again $\mathcal{U}_m\Subset\mathcal{U}_n$ implies $\mathcal{U}_{m^*m}=\mathsf{s}(\mathcal{U}_m)\Subset \mathsf{s}(\mathcal{U}_n)=\mathcal{U}_{n^*n}$ and hence $m^*m<_bn^*n$, for some $b\in B$.  Thus $m^*m=m^*mbn^*n$ so $m=mbn^*n$.  As $mbn^*\in mBn^*\subseteq B$ by \cref{cor:binormal} and the fact that $mn^*\in B$, this implies $m<n$ by \eqref{WeakDomination}.
\end{proof}

\section{Source and Range States}\label{sec:equivalence}

Here we examine how ultrafilters interact with characters on $B$.  In particular, we will see how to precisely measure the magnitude and angle between elements of any given ultrafilter.  These functions will play an important role in the following sections.

To start with, note that each $U\in G$ defines a source ideal $B\setminus \mathsf{s}(U)$ and a range ideal $B\setminus \mathsf{r}(U)$ by \cref{ComplementHomeo}.  These are the kernels (as noted in \Cref{NormedSpaces}) of what we then call the source and range states of $U$, which we denote by
\begin{align}\label{eq:srstates}
\psi_U:=\langle B\setminus \mathsf{s}(U)\rangle\qquad\text{and}\qquad\psi^U:=\langle B\setminus \mathsf{r}(U)\rangle.
\end{align}

\begin{remark}
If $A$ is a $D$-contractive C*-completion of $C_c(\Sigma;G)$ for a twist $q:\Sigma\twoheadrightarrow G$ and $N=\mathrm{cl}(\Nc(\Sigma;G))$ is its monomial semigroup, as in \cref{UltrafiltersArePoints}, these states are given by evaluation at the source and range of the corresponding elements of $\Sigma$, i.e.~for $g=q(e)$,
\[\psi_{U_{g}}(b)=j(b)(\mathsf{s}(e))\qquad\text{and}\qquad\psi^{U_{g}}(b)=j(b)(\mathsf{r}(e)).\]
\end{remark}

We can always calculate the source state of a particular ultrafilter $U$ from its range state and vice versa by using the following formula.

\begin{lemma}\label{prop:psiU}
    For all $U\in G$ and $n\in U$,
    \begin{equation*}\psi^U(b)=\psi_U(n^*bn)/\psi_U(n^*n).
    \end{equation*}
\end{lemma}
\begin{proof}
    If $n\in U$ then $n^*n\in U^*U\subseteq \mathsf{s}(U)$ and hence $\psi_U(n^*n)=\langle B\setminus \mathsf{s}(U)\rangle(n^*n)>0$.  Replacing $n$ with $n/\sqrt{\psi_U(n^*n)}$ if necessary, we may then assume that $\psi_U(n^*n)=1$.  Then note the map $\chi:B \to \mathbb{C}$ defined by $\chi(b)=\psi_U(n^*bn)$ is multiplicative because, for all $b,c\in B$, \[\psi_U(n^*bcn)=\psi_U(n^*bcn)\psi_U(n^*n)=\psi_U(n^*bcnn^*n)=\psi_U(n^*bnn^*cn)=\psi_U(n^*bn)\psi_U(n^*cn).\]
    As characters on $B$ are determined by their kernels, it suffices to show $\mathrm{ker}(\chi)=\mathrm{ker}(\psi^U)$ or even just $\mathrm{ker}(\chi)\subseteq\mathrm{ker}(\psi^U)$, as these kernels are maximal ideals.  To see this just note that $\psi^U(b)\neq0$ means $b\in \mathsf{r}(U)$, so $n^*bn\in U^*\mathsf{r}(U)U\subseteq \mathsf{s}(U)$ and thus $\chi(b)=\psi_U(n^*bn)\neq0$.
\end{proof}

The map to source (and range) states is also continuous in the following sense.

\begin{lemma}\label{StateContinuity}
    If $U_{\lambda}\rightarrow U$ in $G$ and $b_\lambda\rightarrow b$ in $B$ then $\psi_{U_\lambda}(b_\lambda)\rightarrow\psi_U(b)$.
\end{lemma}

\begin{proof}
    As $G$ is \'etale, the source map is continuous so $U_\lambda\rightarrow U$ implies $\mathsf{s}(U_\lambda)\rightarrow \mathsf{s}(U)$.  Then by \cref{ComplementHomeo}, $B\setminus \mathsf{s}(U_\lambda)\rightarrow B\setminus \mathsf{s}(U)$ in the maximal ideal space $X_B$, so $\langle B\setminus \mathsf{s}(U_\lambda)\rangle\rightarrow\langle B\setminus \mathsf{s}(U)\rangle$ in the weak$^*$-topology on states.  This means that $\psi_{U_\lambda}(b_\lambda)\rightarrow\psi_U(b)$.
\end{proof}

\subsection{Magnitudes}

For any $U\in G$, we can define the \emph{$U$-magnitude} of any $n\in U$ by
\[|n|_U:=\sqrt{\psi_U(n^*n)}.\]

\begin{remark}
When $A$ is a $D$-contractive C*-completion of $C_c(\Sigma;G)$ for some twist $q:\Sigma\twoheadrightarrow G$ and $N=\mathrm{cl}(\Nc(\Sigma;G))$ is its monomial semigroup, we saw in \cref{UltrafiltersArePoints} that $G$ is isomorphic to the ultrafilters in $N$ with respect to $<$.  For 
any $n\in N$ and $e\in\Sigma$, the $U_{q(e)}$-magnitude (where $U_{q(e)}$ is the ultrafilter corresponding to $q(e)\in G$) is just the absolute value of $j(n)$ at $e$, i.e.
\[|n|_{U_{q(e)}}=|j(n)(e)|.\]
\end{remark}

\begin{prop}\label{UnormProperties}
    For any composable pair $(T,U)\in G^{(2)}$, $m\in T$, $n\in U$, and $\alpha\in\mathbb{C}\setminus\{0\}$,
    \[|\alpha n|_U=|\alpha||n|_U,\qquad|n^*|_{U^*}=|n|_U\qquad\text{and}\qquad|mn|_{TU}=|m|_T|n|_U.\]
\end{prop}

\begin{proof}
    We immediately see that
    \[|\alpha n|_U=\sqrt{\psi_U((\alpha n)^*(\alpha n))}=\sqrt{\overline{\alpha}\alpha\psi_U( n^*n)}=|\alpha||n|_U.\]
    Also, by \cref{prop:psiU},
    \[|n^*|_{U^*}=\psi^U(nn^*)=\psi_U(n^*nn^*n)/\psi_U(n^*n)=\psi_U(n^*n)\psi_U(n^*n)/\psi_U(n^*n)=|n|_U.\]
    Again \cref{prop:psiU} yields
    \[\hspace{-2pt}|mn|_{TU}=\sqrt{\psi_U(n^*m^*mn)}=\sqrt{\psi^U(m^*m)\psi_U(n^*n)}=\sqrt{\psi_T(m^*m)\psi_U(n^*n)}=|m|_T|n|_U.\qedhere\]
\end{proof}

Often it will suffice to consider elements of $U$-magnitude $1$, which we denote by
\[U_1=\{n\in U:|n|_U=1\}.\]
Note that $\tfrac{1}{|n|_U}n\in U_1$, for all $n\in U$, as $|\frac{1}{|n|_U}n|_U=\frac{1}{|n|_U}|n|_U=1$. 
 In particular, $U_1\neq\emptyset$, for all $U\in G$, which we can strengthen as follows.

\begin{prop}\label{U11}
    For all $U\in G$,
    \[U^1_1:=U_1\cap A^1\neq\emptyset.\]
\end{prop}

\begin{proof}
    Take any $n\in U_1$ and any continuous function $f:\mathbb{R}_+\rightarrow[0,1]$ with $f(0)=0$, $f(1)=1$ and $f(x)\leq1/\sqrt{x}$, for all $x>1$.  For any character $\phi$ on $B$ and any $b\in B_+$ with $\phi(b)=1$, it follows that $\phi(f(b))=1$ as well, as $f(1)=1$.  In particular this holds for $\phi:=\psi_U$ and $b:=n^*n$.  Then $\psi_U(f(b))=1\neq0$ implies $f(b)\in\mathsf{s}(U)$ so $m:=nf(b)\in U\mathsf{s}(U)\subseteq U$ and
    \[\|m\|^2=\|m^*m\|=\|f(b)bf(b)\|=1=\psi_U(f(b)bf(b))=\psi_U(m^*m)=|m|_U^2,\]
    showing that $m\in U^1_1$.
\end{proof}

The following lemma will soon be needed to define angles between ultrafilters.

\begin{lemma}
\label{EmnMag}
    Let $U \in G$. Whenever $m,n\in U$,
    \[|\psi_U(E(m^*n))|=|m|_U|n|_U.\]
\end{lemma}

\begin{proof}
    As $m^*n\in U^*U\subseteq \mathsf{s}(U)$ and hence $E(m^*n)\in \mathsf{s}(U)$ by \cref{lem:units and Un}, we can apply \cref{lem:homom-like} and \cref{UnormProperties} to obtain
    \begin{align*}
    |\psi_U(E(m^*n))|&=\sqrt{\psi_U(E(n^*m)E(m^*n))}=\sqrt{\psi_U(E(n^*mm^*n))}\\
    &=|m^*n|_{\mathsf{s}(U)}=|m^*|_{U^*}|n|_U=|m|_U|n|_U.\qedhere
    \end{align*}
\end{proof}

\subsection{Angles}

Let $U\in G$. Whenever $m,n\in U$, we define the \emph{$U$-angle} from $m$ to $n$ by
\[\langle m,n\rangle_U:=\tfrac{1}{|m|_U|n|_U}\psi_U(E(n^*m)).\]
Note $\lal m,n\ral_U\in\mathbb{T}$ for all $m,n\in U$, by \cref{EmnMag}.

\begin{remark}
When $A$ is a $D$-contractive C*-completion of $C_c(\Sigma;G)$ for some twist $q:\Sigma\twoheadrightarrow G$ and $N=\mathrm{cl}(\Nc(\Sigma;G))$ is its monomial semigroup, we see that
\[\lal m,n\ral_{U_{q(e)}}=j(m)(e)\overline{j(n)(e)}/|j(m)(e)j(n)(e)|,\]
for all $m,n\in N$ and $e\in\Sigma$, where again $U_{q(e)}$ is the ultrafilter corresponding to $q(e)$.
\end{remark}

\begin{prop}\label{AngleProperties}
Let $U \in G$. Whenever $l,m,n\in U$,
\[\lal n,n\ral_U=1,\qquad \lal m^*,n^*\ral_{U^*}=\lal n,m\ral_U=\overline{\lal m,n\ral_U}\qquad\text{and}\qquad\lal l,n\ral_U=\lal l,m\ral_U\lal m,n\ral_U.\]
\end{prop}

\begin{proof}
    As $\lal v,w\ral_U=\lal\frac{1}{|v|_U}v,\frac{1}{|w|_U}w\ral_U$ for all $v,w\in U$, we may assume $l,m,n\in U_1$.  Then
    \[\lal n,n\ral_U=\psi_U(E(n^*n))=\psi_U(n^*n)=1.\]
    Also \cref{prop:Normal} and \cref{prop:psiU} yield
    \begin{align*}   
    \lal m^*,n^*\ral_{U^*}&=\psi^U(E(nm^*))=\psi_U(n^*E(nm^*)n)=\psi_U(E(n^*nm^*n))=\psi_U(n^*n)\psi_U(E(m^*n))\\
    &=\lal n,m\ral_U=\psi_U(E(m^*n))=\overline{\psi_U(E(n^*m))}=\overline{\lal m,n\ral_U}.
    \end{align*}
    Applying \cref{prop:Normal} and \cref{prop:psiU} again, this time with \cref{lem:homom-like} too yields
    \begin{align*}        
    \lal l,m\ral_U \lal m,n\ral_U&=\psi_U(E(m^*l)E(n^*m))=\psi_U(E(m^*ln^*m))=\psi^U(E(ln^*))\\
    &=\lal n^*,l^*\ral_{U^*}=\lal l,n\ral_U.\qedhere
    \end{align*}
\end{proof}

Angles also respect products in the following sense.

\begin{prop}\label{AngleProducts}
Let $U \in G$. Whenever $(U,V)\in G^{(2)}$, $m,n\in U$, and $r, s\in V$,
\[\lal m,n\ral_U \lal r,s\ral_V=\lal mr,ns\ral_{U\cdot V}.\]
\end{prop}

\begin{proof}
    Scaling if necessary, and applying \cref{UnormProperties} we may assume $m,n\in U_1$ and $r,s\in V_1$.  By \cref{prop:Normal} and \cref{prop:psiU},
    \[\lal mr,nr\ral_{U\cdot V}=\psi_V(E(r^*n^*mr))=\psi_V(r^*E(n^*m)r)=\psi_U(E(n^*m))=\lal m,n\ral_U.\]
    Likewise, $\lal nr,ns\ral_{U\cdot V}=\lal r,s\ral_V$ and so \cref{AngleProperties} yields
    \[\lal m,n\ral_U \lal r,s\ral_V= \lal mr,nr\ral_{U\cdot V}\lal nr,ns\ral_{U\cdot V}= \lal mr,ns\ral_{U\cdot V}.\qedhere\]
\end{proof}

We also have the following observations for products from scalar multiplication and unit ultrafilters.

\begin{lemma}
    Let $U \in G$. Whenever $n\in U$, $\alpha\in\mathbb{C}\setminus\{0\}$, and $l\in\mathsf{r}(U)$,
    \[\lal\alpha n,n\ral_U= \alpha/|\alpha|=\lal n,\overline{\alpha}n\ral_U\qquad\text{and}\qquad \lal ln,n\ral_U= \psi^U(E(l))/|l|_{\mathsf{r}(U)}= \lal n,l^*n\ral_U.\]
\end{lemma}

\begin{proof}
    We may assume $n\in U_1$, in which case $\lal\alpha n,n\ral_U=\tfrac{1}{|\alpha n|_U}\psi_U(\alpha n^*n)=\alpha/|\alpha|$.  Then we also see that $\lal n,\overline{\alpha}n\ral_U=\overline{\lal\overline{\alpha}n,n\ral_U}=\overline{\overline{\alpha}/|\alpha|}=\alpha/|\alpha|$.  Similarly, assuming $l\in\mathsf{r}(U)_1$, \cref{prop:Normal} and \cref{prop:psiU} again yield $\lal ln,n\ral_U=\psi_U(E(n^*ln))=\psi^U(E(l))$ and hence $\lal n,l^*n\ral_U=\overline{\lal l^*n,n\ral_U}=\overline{\psi^U(E(l^*))}=\psi^U(E(l))$ as well.
\end{proof}

Likewise, we also see that, whenever $n\in U$ and $l\in\mathsf{s}(U)$,
\[\lal nl,n\ral_U=\psi_U(E(l))/|l|_{\mathsf{s}(U)}=\lal n,nl^*\ral_U.\]
In particular, for any $b\in\mathsf{r}(U)_+(=\mathsf{r}(U)\cap B_+)$ and $c\in\mathsf{s}(U)_+$,
\begin{equation}\label{bcInvariant}
    \lal bn,n\ral_U= \lal n,bn\ral_U=1= \lal nc,n\ral_U=\lal n,nc\ral_U.
\end{equation}

\subsection{Equivalences}

We define a binary relation $\sim_U$ on each $U\in G$, which we will soon see that it is an equivalence relation.  In \cref{sec:sigma}, the resulting equivalence classes will form the domain $\Sigma$ of our twist $q:\Sigma\twoheadrightarrow G$.
\begin{definition} 
For each $U\in G$, whenever $m,n\in U$, we write $m\sim_Un$ to mean $\lal m,n\ral_U=1$ or, equivalently, $\psi_U(E(n^*m))>0$ so
\[m\sim_Un\qquad\Leftrightarrow\qquad\lal m,n\ral_U=1\qquad\Leftrightarrow\qquad\psi_U(E(n^*m))>0.\]
    \end{definition}

\begin{remark}
\label{rmk:ufequiv}
Continuing the example from \cref{UltrafiltersArePoints}, let $A$ be a $D$-contractive C*-completion of $C_c(\Sigma;G)$ for some twist $q:\Sigma\twoheadrightarrow G$ and $N=\mathrm{cl}(\Nc(\Sigma;G))$ be its monomial semigroup, we see that
\[m\sim_{U_{q(e)}}n\qquad\Leftrightarrow\qquad j(m)(e)/|j(m)(e)|=j(n)(e)/|j(n)(e)|,\]
for all $m,n\in N$ and $e\in\Sigma$, where again $U_{q(e)}$ is the ultrafilter corresponding to $q(e)$.
\end{remark}

\begin{prop}\label{SimProperties}
    For each $U\in G$, $\sim_U$ is an equivalence relation such that $m\sim_Un$ implies $m^*\sim_{U^*}n^*$ and $\alpha m\sim_U\alpha n$, for all $\alpha\in\mathbb{C}\setminus\{0\}$.  Moreover, if $(U,V)\in G^{(2)}$ then
    \begin{equation}\label{SimProducts}
        m\sim_Un\quad\text{and}\quad q\sim_Vr\qquad\Rightarrow\qquad mq\sim_{U\cdot V}nr.
    \end{equation}
\end{prop}

\begin{proof}
    For each $U\in G$, we apply \cref{AngleProperties} to show that $\sim_U$ is an equivalence relation.  Firstly, $\sim_U$ is reflexive because $\lal n,n\ral_U=1$ for all $n\in N$.  As $\lal m,n\ral_U=1$ implies $\lal n,m\ral_U=\overline{\lal m,n\ral_U}=1$, we see that $\sim_U$ is symmetric.  As $\lal l,m\ral_U=1$ and $\lal m,n\ral_U=1$ implies $\lal l,n\ral_U= \lal l,m\ral_U \lal m,n\ral_U=1$, we see that $\sim_U$ is also transitive.  Likewise, $\lal m,n\ral_U=1$ implies $\lal m^*,n^*\ral_{U^*}=\overline{\lal m,n\ral_U}=1$ and
    \[\lal\alpha m,\alpha n\ral_U=\lal\alpha m,m\ral_U \lal m,n\ral_U \lal n,\alpha n\ral_U=\alpha\overline{\alpha}/|\alpha|^2=1,\] showing that $m\sim_Un$ implies $m^*\sim_{U^*}n^*$ and $\alpha m\sim_U\alpha n$, for all $\alpha\in\mathbb{C}\setminus\{0\}$.  Likewise, \eqref{SimProducts} is immediate from \cref{AngleProducts}.
\end{proof}

Accordingly, if $m\sim_Un$ then we say that $m$ and $n$ are \emph{$U$-equivalent}.  All elements of any given ultrafilter $U$ are in fact $U$-equivalent modulo a unique factor in $\mathbb{T}$, namely $\lal m,n\ral_U$.  Indeed, if $m,n\in U\in G$ and $t\in\mathbb{T}$ then
\begin{equation}\label{msimtn}
    m\sim_Utn\qquad\Leftrightarrow\qquad 1=\lal m,tn\ral_U=\overline{t}\lal m,n\ral_U\qquad\Leftrightarrow\qquad t=\lal m,n\ral_U.
\end{equation}
Also, products with positive elements in the source or range of $U$ are $U$-equivalent, i.e.
\[b\in \mathsf{r}(U)_+,\quad c\in \mathsf{s}(U)_+\quad\text{and}\quad n\in U\qquad\Rightarrow\qquad bn\sim_Un\sim_Unc,\]
by \eqref{bcInvariant}.  In particular, positive elements in any $U\in G^{(0)}$ are always $U$-equivalent, as
\[b,c\in U_+\qquad\Rightarrow\qquad b\sim_Ubc\sim_Uc.\]

Let us now denote the $U$-equivalence class of any $n\in U$ by
\[[n]_U=\{m\in U:m\sim_Un\}.\]
The ultrafilter $U$ can always be recovered from any of its $U$-equivalence classes as follows.

\begin{prop}\label{URecovery}
    Let $U \in G$. Whenever $n\in U$,
    \[U=[n]_U^<.\]
\end{prop}

\begin{proof}
    Assume $n\in U$.  We immediately see that $[n]_U\subseteq U$ and hence $[n]_U^<\subseteq U^<\subseteq U$ since $U$ is an ultrafilter.  Conversely, for any $m\in U$, we have $l\in U$ with $l<m,n$ and then $n\sim_U \lal n,l\ral_U l<m$ so $m\in[n]_U^<$.  This shows that $U\subseteq[n]_U^<$.
\end{proof}

\section{The Twist}\label{sec:sigma}

In this section, we investigate a natural twist over our ultrafilter groupoid $G$ formed from all its ultrafilter equivalence classes.  We denote these equivalence classes by
\[\Sigma=\Sigma_N:=\{[n]_U: \: n\in U, U\in G\},\]
which we consider as a topological space with the subbasis
\begin{align*}
    \mathcal{B}&=\{\mathcal{E}^O_n:O\subseteq\mathbb{T}\text{ is open and }n\in N\}\text{ where}\\
    \mathcal{E}^O_n&=\{[tn]_U:t\in O\text{ and }U\in\mathcal{U}_n\}.
\end{align*}
Thanks to \eqref{msimtn} above and \cref{URecovery}, we have
\[[m]_U\in\mathcal{E}^O_n\qquad\Leftrightarrow\qquad \lal m,n\ral_U\in O,\]
where $U\in\mathcal{U}_n$ is implicit on the right hand side for $\lal m,n\ral_U$ to be defined.

We first claim that $\mathcal{B}$ above is not just a subbasis for the topology, it is actually a basis.  In fact, we will prove the stronger result that each point $[m]_U$ has a special neighbourhood base consisting of open sets of the form $\mathcal{E}^O_n$, where $O$ is a neighbourhood of $1$ and $n$ is a multiple of $m$ with some positive element of $B$.

\begin{prop}\label{NeighbourhoodBase}
    Every $[m]_U\in\Sigma$ has a neighbourhood base of the form
    \[\{\mathcal{E}^O_{bm}:b\in \mathsf{r}(U)_+\text{ and $O$ is a neighbourhood of $1$ in }\mathbb{T}\}.\]
\end{prop}

\begin{proof}
    Let $P \subseteq \mathbb{T}$ be open. If $[m]_U\in\mathcal{E}^P_n$ and hence $\lal m,n\ral_U\in P$ then, as $V\mapsto\lal m,n\ral_V=\tfrac{1}{|m|_V|n|_V}\psi_V(E(n^*m))$ is continuous, we must have $l\in U$ with $l<m,n$ such that $C=\mathrm{cl}\{\lal m,n\ral_V:V\in\mathcal{U}_l\}\subseteq P$.  For a sufficiently small neighbourhood $O$ of $1$ in $\mathbb{T}$, we then still have $OC\subseteq P$.  Letting $b=ll^*\in \mathsf{r}(U)_+$, we immediately see that $\lal m,bm\ral_U=1\in O$ so $[m]_U\in\mathcal{E}^O_{bm}$.  If $[q]_V\in\mathcal{E}^O_{bm}$, then we also see that
    \[V\in\mathcal{U}_{bm}=\mathcal{U}_l\mathcal{U}_{l^*m}\subseteq\mathcal{U}_lG^{(0)}\subseteq\mathcal{U}_l\subseteq\mathcal{U}_n\]
    and $\lal q,m\ral_V= \lal q,bm\ral_V\in O$ so $\lal q,n\ral_V=\lal q,m\ral_V \lal m,n\ral_V\in OC\subseteq P$ and hence $[q]_V\in\mathcal{E}^P_n$, showing that $\mathcal{E}^O_{bm}\subseteq\mathcal{E}^P_n$.

    So we have shown that open sets of the given form are a neighbourhood subbase at $[m]_U\in\Sigma$.  To see that they are even a neighbourhood base, it suffices to show that
    \[\mathcal{E}^O_{bm}\cap\mathcal{E}^P_{cm}=\mathcal{E}^{O\cap P}_{bcm}.\]
    To see this, first note that $\mathcal{U}_{bcm}=\mathcal{U}_b\mathcal{U}_c\mathcal{U}_m=(\mathcal{U}_b\cap\mathcal{U}_c)\mathcal{U}_m=\mathcal{U}_{bm}\cap\mathcal{U}_{cm}$.  We then immediately see that $\mathcal{E}^{O\cap P}_{bcm}\subseteq\mathcal{E}^O_{bm}\cap\mathcal{E}^P_{cm}$.  Conversely, if $[l]_U\in\mathcal{E}^O_{bm}\cap\mathcal{E}^P_{cm}$ then $\lal l,m\ral_U=\lal l,bm\ral_U\in O$ and $\lal l,m\ral_U=\lal l,cm\ral_U\in P$ so $\lal l,bcm\ral_U=\lal l,m\ral_U\in O\cap P$ and hence $[l]_U\in\mathcal{E}^{O\cap P}_{bcm}$. 
\end{proof}

Likewise, every $[m]_U\in\Sigma$ has a neighbourhood base of the form
\[\{\mathcal{E}^O_{mb}:b\in \mathsf{s}(U)_+\text{ and $O$ is a neighbourhood of $1$ in }\mathbb{T}\}.\]
We can also characterise convergence in $\Sigma$ as follows.

\begin{lemma}\label{SigmaLimits}
    For any net $([n_\lambda]_{U_\lambda})\subseteq\Sigma$ and any $[n]_U\in\Sigma$, $[n_\lambda]_{U_\lambda}\rightarrow[n]_U$ if and only if $U_\lambda\rightarrow U$ and $\lal n_\lambda,n\ral_{U_\lambda}\rightarrow 1$.
  \end{lemma}
  
    We note that the statement of \cref{SigmaLimits} here makes sense: if $U_\lambda\rightarrow U\in\mathcal{U}_n$, then eventually $n\in U_\lambda$ which, in particular, implies $\lal n_\lambda,n\ral_{U_\lambda}$ is defined.  Equivalently, we could have said $[n_\lambda]_{U_\lambda} \to [n]_U$ if and only if $U_\lambda \to U$ and  there exists $t_\lambda \in \T$ such that  $[n_\lambda]_{U_\lambda} = [t_\lambda n]_{U_\lambda}$ eventually where $t_\lambda \to 1$ by \eqref{msimtn}.

\begin{proof}
    Assume $[n_\lambda]_{U_\lambda}\rightarrow[n]_U$.  For every $m\in U$, we see that $[n]_U\in\mathcal{E}^\mathbb{T}_m$ and so eventually $[n_\lambda]_{U_\lambda}\in\mathcal{E}^\mathbb{T}_m$ and hence $U_\lambda\in\mathcal{U}_m$, showing that $U_\lambda\rightarrow U$.  For every open neighbourhood $O$ of $1$ in $\mathbb{T}$, we also see that $[n]_U\in\mathcal{E}^O_n$ and so eventually $[n_\lambda]_{U_\lambda}\in\mathcal{E}^O_n$ and hence $\lal n_\lambda,n\ral_{U_\lambda}\in O$, showing that $\lal n_\lambda,n\ral_{U_\lambda}\rightarrow1$.

    Conversely, say $U_\lambda\rightarrow U$ and $\lal n_\lambda,n\ral_{U_\lambda}\rightarrow 1$.  By \cref{StateContinuity}, $\lal n,m\ral_{U_\lambda}\rightarrow \lal n,m\ral_U$ and hence $\lal n_\lambda,m\ral_{U_\lambda}= \lal n_\lambda,n\ral_{U_\lambda}\lal n,m\ral_{U_\lambda}\rightarrow \lal n,m\ral_U$.  If $[n]_U\in\mathcal{E}^O_m$ and hence $\lal n,m\ral_U\in O$ then this means eventually $\lal n_\lambda,m\ral_{U_\lambda}\in O$ and hence $[n_\lambda]_{U_\lambda}\in\mathcal{E}^O_m$, showing that $[n_\lambda]_{U_\lambda}\rightarrow[n]_U$.
\end{proof}

Under a natural product operation, $\Sigma$ becomes a topological groupoid.

\begin{prop}\label{prop:sigma}
    The space $\Sigma$ is a Hausdorff topological groupoid under the product
    \[[m]_T[n]_U:=[mn]_{T\cdot U}\text{ defined if and only if }(T,U)\in G^{(2)},\]
    where the inverse of each $[n]_U\in\Sigma$ is given by $[n]_U^{-1}=[n^*]_{U^*}$, and the units are given by
    \[\Sigma^{(0)}=\{[b]_U:b\in B_+\text{ and }U\in\mathcal{U}_b\}.\]
\end{prop}

\begin{proof}
    If $[m]_T=[m']_{T'}$ and $[n]_U=[n']_{U'}$ then $T=T'$ and $U=U'$, by \cref{URecovery}, so $(T,U)\in G^{(2)}$ if and only if $(T',U')\in G^{(2)}$.  In this case, $m\sim_Tm'$ and $n\sim_Un'$ and hence $mn\sim_{T\cdot U}m'n'$ by \cref{SimProperties}, i.e. $[mn]_{T\cdot U}=[m'n']_{T'\cdot U'}$.  Thus the product is well-defined and associativity follows from the associativity of the product in $N$.
    
    Now if $b\in B_+$ and $U\in\mathcal{U}_b$ then, for any $[m]_T\in\Sigma$ with $\mathsf{s}(T)=U$, we see that $[m]_T[b]_U=[mb]_{T\cdot U}=[m]_T$, as $m\sim_T mb$.  Likewise, $[b]_U[m]_T=[m]_T$ for any $[m]_T\in\Sigma$ with $\mathsf{r}(T)=U$, so $[b]_U$ is a unit in $\Sigma$.  In particular, every $[n]_U\in\Sigma$ has a source unit $\mathsf{s}([n]_U)=[n^*n]_{\mathsf{s}(U)}=[n^*]_{U^*}[n]_U$ and range unit $\mathsf{r}([n]_U)=[nn^*]_{\mathsf{r}(U)}=[n]_U[n^*]_{U^*}$.  This shows that $\Sigma$ is a groupoid with units $\Sigma^{(0)}=\{[b]_U:b\in B_+\text{ and }U\in\mathcal{U}_b\}$ and inverse operation $[n]_U\mapsto[n^*]_{U^*}$.

    To see that $\Sigma$ is a topological groupoid, we must show that the product and inverse operations are continuous.  Accordingly, note that if $(T_\lambda,U_\lambda)\subseteq G^{(2)}$, $[m_\lambda]_{T_\lambda}\rightarrow[m]_T$ and $[n_\lambda]_{U_\lambda}\rightarrow[n]_U$ in $\Sigma$ then, by \cref{SigmaLimits}, $T_\lambda\rightarrow T$, $\lal m_\lambda,m\ral_{T_\lambda}\rightarrow 1$, $U_\lambda\rightarrow U$ and $\lal n_\lambda,n\ral_{U_\lambda}\rightarrow 1$.  Then $T_\lambda\cdot U_\lambda\rightarrow T\cdot U$, by the continuity of the product in $G$, and $\lal m_\lambda n_\lambda,mn\ral_{T_\lambda\cdot U_\lambda}=\lal m_\lambda,m\ral_{T_\lambda}\lal n_\lambda,n\ral_{U_\lambda}\rightarrow 1$.  Again by \cref{SigmaLimits}, this means $[m_\lambda]_{T_\lambda}[n_\lambda]_{U_\lambda}=[m_\lambda n_\lambda]_{T_\lambda\cdot U_\lambda}\rightarrow[mn]_{T\cdot U}=[m]_T[n]_U$, showing that the product is continuous.  But  $T_\lambda^*\rightarrow T^*$ by the continuity of the inverse in $G$ and $\lal m_\lambda^*,m^*\ral_{T_\lambda^*}=\overline{\lal m_\lambda,m\ral_{T_\lambda}}\rightarrow 1$.  Then \cref{SigmaLimits} again implies that $[m_\lambda^*]_{T_\lambda^*}\rightarrow[m^*]_{T^*}$, showing that the inverse in $\Sigma$ is also continuous and hence $\Sigma$ is a topological groupoid.

    Finally, to see that $\Sigma$ is Hausdorff, say we have a net $([n_\lambda]_{U_\lambda})$ with limits $[m]_T$ and $[n]_U$ in $\Sigma$.  By \cref{SigmaLimits}, this means $U_\lambda\rightarrow T$, $U_\lambda\rightarrow U$, $\lal n_\lambda,m\ral_{U_\lambda}\rightarrow 1$ and $\lal n_\lambda,n\ral_{U_\lambda}\rightarrow 1$.  As $G$ is Hausdorff, $T=U$.  Also $\lal m,n\ral_U=\lim_\lambda\lal m,n\ral_{U_\lambda}=\lim_\lambda\lal m,n_\lambda\ral_{U_\lambda}\lal n_\lambda,n\ral_{U_\lambda}=1$ and hence $[m]_U=[n]_U$.  Thus limits in $\Sigma$ are unique and hence $\Sigma$ is Hausdorff.
\end{proof}

\begin{remark}\label{rmk:recover}
Suppose $A$ is a $D$-contractive C*-completion of $C_c(\Sigma;G)$ for some twist $q:\Sigma\twoheadrightarrow G$ and $N=\mathrm{cl}(\Nc(\Sigma;G))$ is its monomial semigroup. We claim there is a topological groupoid isomorphism from the original $\Sigma$ onto $\Sigma_N$ defined by 
\[e\mapsto N_e:=\{n\in N:j(n)(e)>0\}.\]
To see this map takes values in $\Sigma_N$, note that  for any $n \in N_e$, 
$N_e = [n]_{U_{q(e)}}$ 
 by \cref{rmk:ufequiv}.  It is straightforward to check that the map is a groupoid homomorphism. That the map is bijective follows from the bijectivity of the map from $G$ to $G_N$ given in \cref{UltrafiltersArePoints} along with choices of sections from $G$ to $\Sigma$. 
 Continuity follows from \cref{SigmaLimits}, the continuity in \cref{UltrafiltersArePoints} and the continuity of each function $j(n)$.  
 Finally, \cref{SigmaLimits} and \cref{UltrafiltersArePoints} also imply the inverse is continuous. 
Thus the  equivalence classes $\Sigma_N$ `recover' the original $\Sigma$.  
\end{remark}

Next we observe that $G$ is both a topological and algebraic quotient of $\Sigma$.

\begin{prop}\label{qContinuousOpen}
    There is a continuous open groupoid homomorphism $q$ from $\Sigma$ onto $G$ defined by
    \[q([n]_U):=U.\]
\end{prop}

\begin{proof}
    By \cref{URecovery}, $q$ is well-defined and, since ultrafilters are nonempty, $q$ maps $\Sigma$ onto $G$.  By the definition of the product on $\Sigma$, $q$ is a groupoid homomorphism.  As $q^{-1}(\mathcal{U}_n)=\mathcal{E}^\mathbb{T}_n$, for all $n\in N$, $q$ is also continuous.  As $\mathcal{B}$ is a basis and $q(\mathcal{E}^O_n)=\mathcal{U}_n$, for all $n\in N$ and open $O\subseteq\mathbb{T}$, we see that $q$ is also an open map.
\end{proof}
    
We have a natural action of $\mathbb{T}$ on $\Sigma$.

\begin{prop}\label{TAction}
There is a free continuous open action of $\mathbb{T}$ on $\Sigma$ given by
\[t[n]_U:=[tn]_U.\]
Moreover, $\mathbb{T}$ acts transitively on each fibre of $q$.
\end{prop}

\begin{proof}
    If $m\sim_Un$ and $t\in\mathbb{T}$ then $tm\sim_Utn$, so the action is well-defined.  And if $sn\sim_Utn$, for some $s,t\in\mathbb{T}$ and $n\in N$, then $s=t$ so the action is free.  For any open $O,P\subseteq\mathbb{T}$ and $n\in N$, we see that $O\mathcal{E}^P_n=\mathcal{E}^{OP}_n$, so the action is also open. If we have nets $t_\lambda\rightarrow t$ in $\mathbb{T}$ and $[n_\lambda]_{U_\lambda}\rightarrow[n]_U$ in $\Sigma$ then $U_\lambda\rightarrow U$ and $\lal n_\lambda,n\ral_U\rightarrow 1$, by \cref{SigmaLimits}, so $\lal t_\lambda n_\lambda,tn\ral_U=t_\lambda\overline{t}\lal n_\lambda,n\ral_U\rightarrow 1$ and hence $t_\lambda[n_\lambda]_{U_\lambda}=[t_\lambda n_\lambda]_{U_\lambda}\rightarrow[tn]_U=t[n]_U$, again by \cref{SigmaLimits}.  This shows that the action is continuous.  Transitivity on the fibres of $q$ follows from the fact that $[m]_U=\lal m,n\ral_U[n]_U$, for all $m,n\in U$.
\end{proof}

Putting together \cref{prop:sigma}, \cref{qContinuousOpen} and \cref{TAction}, we have the following. 

\begin{cor}
    With $G, \Sigma$ and $q$ as defined above, $q:\Sigma\twoheadrightarrow G$ is a twist.
\end{cor}

\section{The Representation}\label{sec:rep}

Here we show how to represent elements of $A$ as functions in $C_0(\Sigma;G)$, yielding an isomorphism of $A$ with a twisted groupoid C*-algebra.  This achieves our main goal of showing that, up to isomorphism, twisted groupoid C*-algebras are completely characterised by having Cartan semigroups.  Moreover, we will show that summable Cartan semigroups are exactly those that can be identified with the monomial semigroups of twisted groupoid C*-algebras, thus providing a precise converse to \cref{TwistedGroupoid->CartanSemigroup}.

\begin{remark}
    In the process we will see that $E$ is automatically faithful on a dense *-subalgebra of $A$ but, unlike elsewhere in the literature, we do not require it to be faithful everywhere.  If it is faithful on all of $A$, however, then the twisted groupoid C*-algebra isomorphic to $A$ is indeed reduced, just like in the original Kumjian-Renault theory.
\end{remark}

For every $a\in A$, we define the desired function $\hat{a}$ as follows.

\begin{prop}\label{prop:hat a}
For every $a\in A$, we have $\hat{a}\in C(\Sigma;G)$ defined by
\begin{equation}\label{def:hat}
\hat{a}([n]_U):=\tfrac{1}{|n|_U}\psi_U(E(n^*a))=\tfrac{1}{|n|_U}\psi^U(E(an^*)).
\end{equation}
\end{prop}

\begin{proof}
    Note this amounts to saying that, for $n\in U_1$,
    \[\hat{a}([n]_U):=\psi_U(E(n^*a))=\psi^U(E(an^*)).\]
    To see that this is well-defined, take any $m,n\in U_1$ with $m\sim_Un$. Then \cref{prop:Normal}, \cref{lem:homom-like}, \cref{prop:psiU} and \cref{EmnMag} yield
    \begin{align*}
        \psi_U(E(m^*a))&=\psi^U(nE(m^*a)n^*)=\psi^U(E(nm^*an^*))=\psi^U(E(nm^*)E(an^*))\\
        &=\psi^U(E(an^*)).
    \end{align*}
    To see that $\hat{a}$ is $\mathbb{T}$-contravariant, for any $U\in G$, $n\in U_1$ and $t\in\mathbb{T}$ we have
    \[\hat{a}(t[n]_U)=\hat{a}([tn]_U)=\psi_U(E(\overline{t}n^*a))= \overline{t}\psi_U(E(n^*a))= \overline{t}\hat{a}([n]_U).\]
    To verify that $\hat{a}$ is continuous, take any nets $(U_\lambda)_{\lambda\in \Lambda}\subseteq G$ and $(m_\lambda)_{\lambda \in \Lambda}$ with $m_\lambda\in (U_{\lambda})_1$, for all $\lambda\in \Lambda$, such that $[m_\lambda]_{U_\lambda} \to [m]_U$ in $\Sigma$, for some $U\in G$ and $m\in U_1$.  This means $U_\lambda\to U$ and $\lal m_\lambda,m\ral_{U_\lambda}\rightarrow 1$ by \cref{SigmaLimits}.  It follows that, for all sufficiently large $\lambda$,
    \[\hat{a}([m_\lambda]_{U_\lambda})=\hat{a}(\lal m_\lambda,m\ral_{U_\lambda}[m]_{U_\lambda})=\overline{\lal m_\lambda,m\ral_{U_\lambda}}\psi_{U_\lambda}(m^*a)\rightarrow\psi_U(m^*a)=\hat{a}([m]_U).\qedhere\]
\end{proof}

Now we proceed to examine further properties of the map $a\mapsto\hat{a}$.  Note below we are considering the * operation on $\ell^\infty(\Sigma)$ from \Cref{TwistedGroupoidCstarAlgebras}, where $a^*(e)=\overline{a(e^{-1})}$.

\begin{prop}\label{hatLinear}
    The map $a\mapsto\hat{a}$ is a *-linear contraction from $A$ to $\ell^\infty(\Sigma)$.
\end{prop}

\begin{proof}
    For any $a\in A$, $U\in G$ and $n\in U_1$, note that
    \[\widehat{a^*}([n]_U)=\psi_U(E(n^*a^*))=\overline{\psi^{U^*}(E(an))}=\overline{\hat{a}([n^*]_{U^*})}=\overline{\hat{a}([n]_U^{-1})},\]
    showing that $\widehat{a^*}=\hat{a}^*$.  For any $\alpha\in\mathbb{C}$, we see that
    \[\widehat{\alpha a}([n]_U)=\psi_U(E(\alpha n^*a))=\alpha\psi_U(E(n^*a))=\alpha\hat{a}([n]_U),\]
    showing that $\widehat{\alpha a}=\alpha\hat{a}$.  For any other $b\in A$, we also see that
    \[\widehat{a+b}([n]_U)=\psi_U(n^*(a+b))=\psi_U(n^*a)+\psi_U(n^*b)=\hat{a}([n]_U)+\hat{b}([n]_U),\]
    showing that $\widehat{a+b}=\hat{a}+\hat{b}$.  Finally, for any $U\in G$, \cref{U11} yields $n\in U^1_1$ so
    \[|a([n]_U)|=|\psi_U(E(n^*a))|\leq\|E(n^*a)\|\leq\|n^*a\|\leq\|n\|\|a\|=\|a\|.\]
    This shows that $\|\hat{a}\|_\infty\leq\|a\|$ so $a\mapsto\hat{a}$ is indeed a contraction.
\end{proof}

When we consider the map $a\mapsto \hat{a}$ on $B$, we get a stronger property. First let
\[B_0(\Sigma;G):=\{f\in C_0(\Sigma;G):q(\mathrm{supp}^\circ(f))\subseteq G^{(0)}\} \supseteq \Bc(\Sigma;G).\]

\begin{prop}\label{Bhat}
    The map $a\mapsto\hat{a}$ is a C*-algebra isomorphism from $B$ onto $B_0(\Sigma;G)$.
\end{prop}

\begin{proof}
    This follows from the classic Gelfand duality.  To see this, first note that we have a C*-algebra isomorphism from $B_0(\Sigma;G)$ onto $C_0(G^{(0)})$, specifically given by $f\mapsto f\circ q|_{\Sigma^{(0)}}^{-1}$, where $q|_{\Sigma^{(0)}}^{-1}$ is the inverse of the quotient map restricted to $\Sigma^{(0)}$.  For each $U\in G^{(0)}$, take $b_U\in U_1\cap B_+$ so $q|_{\Sigma^{(0)}}^{-1}(U)=[b_U]_U$.  By \cref{ComplementHomeo}, we have a homeomorphism $u:X_B\twoheadrightarrow G^{(0)}$ (namely $u(I)=(B\setminus I)^<$) which again yields a C*-algebra isomorphism from $C_0(G^{(0)})$ onto $C_0(X_B)$ given by $f\mapsto f\circ u$.  For any $a\in B$, we then see that
    \[\hat{a}\circ q|_{\Sigma^{(0)}}^{-1}\circ u(I)=\hat{a}([b_{u(I)}]_{u(I)})=\psi_{u(I)}(b_{u(I)}^*a)=\langle I\rangle(a).\]
    Gelfand duality then tells us that $a\mapsto\hat{a}\circ q|_{\Sigma^{(0)}}^{-1}\circ u$ is a C*-algebra isomorphism from $B$ onto $C_0(X_B)$.  Inverting the isomorphisms above, it follows that $a\mapsto\hat{a}$ is a C*-algebra isomorphism from $B$ onto $B_0(\Sigma;G)$.
\end{proof}

Next we show that composing the stable expectation $E$ with the map $ a \mapsto \hat{a}$ is the same as first mapping to $\hat{a}$ then applying the usual groupoid expectation $\hat{E}$ following the same formula as the diagonal map $D$.  Specifically, for any $f:\Sigma\rightarrow\mathbb{C}$, we define $\hat{E}(f):\Sigma\rightarrow\mathbb{C}$ by
\[\hat{E}(f)(e)=\begin{cases}f(e)&\text{if }q(e)\in G^{(0)}\\0&\text{otherwise.}\end{cases}\]

\begin{lemma}\label{ExpectationRestriction}
    For all $a\in A$,
    \[\widehat{E(a)}=\hat{E}(\hat{a}).\]
\end{lemma}

\begin{proof}
    Take $a\in A$.  For any $U\in G^{(0)}$ and $b\in U_1\cap B$,
    \[\widehat{E(a)}([b]_U)=\psi_U(E(b^*E(a)))=\psi_U(E(E(b^*a)))=\psi_U(E(b^*a))=\widehat{a}([b]_U),\]
    showing that $\widehat{E(a)}$ agrees with $\hat{a}$ on $q^{-1}(G^{(0)})$.  On the other hand, for any $U\in G\setminus G^{(0)}$, \cref{prop:G0complement} yields $n\in U_1$ with $E(n)=0$ and then
    \[\widehat{E(a)}([n]_U)=\psi_U(E(n^*E(a)))=\psi_U(E(n^*)E(a))=0.\]
    Extending to $\mathbb{T}$-multiplies yields $\widehat{E(a)}(q^{-1}(G\setminus G^{(0)}))=\{0\}$, as required.
\end{proof}

Whenever $n\in U\in G$, we see that
\[\hat{n}([n]_U)=\tfrac{1}{|n|_U}\psi_U(n^*n)=|n|_U.\]
This observation allows us to identify $\supp^\circ(\hat{n})$.

\begin{prop}\label{prop:supports}
For any $n\in N$,
\[\supp^\circ(\hat{n})=q^{-1}(\mathcal{U}_n).\]
\end{prop}

\begin{proof}
    For any $U\in\mathcal{U}_n$ and $t\in\mathbb{T}$, we see that $\hat{n}(t[n]_U)=\overline{t}|n|_U\neq0$ so $q^{-1}(\mathcal{U}_n)\subseteq\supp^\circ(\hat{n})$.  Conversely, take $[m]_U\in\supp^\circ(\hat{n})$ so \[
    0\neq\hat{n}([m]_U)=\psi_U(E(m^*n))=\langle B\setminus \mathsf{s}(U)\rangle(E(m^*n)).\]
    This means $E(m^*n)\in \mathsf{s}(U)$ and hence $\mathsf{s}(U)\in\mathcal{U}_{E(m^*n)}\subseteq\mathcal{U}_{m^*n}$ by \cref{lem:units and Un}.  So $m^*n\in \mathsf{s}(U)$ and hence $mm^*n\in U\mathsf{s}(U)\subseteq U$, i.e. $U\in\mathcal{U}_{mm^*n}=\mathcal{U}_{mm^*}\mathcal{U}_n\subseteq\mathcal{U}_n$.  As $U$ was arbitrary, this shows that $\supp^\circ(\hat{n})\subseteq q^{-1}(\mathcal{U}_n)$.
\end{proof}

Building on this, we have the following.

\begin{prop}\label{prop:semihomo}
The map $n\mapsto\hat{n}$ is a semigroup homomorphism from $N$ to $S:=\{a \in C(\Sigma;G) : \: q(\supp^\circ(a)) \text{ is a bisection}\}$.
\end{prop}

\begin{proof}
    For each $n\in N$, \cref{prop:hat a} says $\hat{n}\in C(\Sigma;G)$ and \cref{prop:supports} tells us that $q(\mathrm{supp}^\circ(\hat{n}))=\mathcal{U}_n$ is a bisection and hence $\hat{n}\in S$.  Now for any $(T,U)\in G^{(2)}$, $m\in T$ and $n\in U$,
    \[\widehat{mn}([m]_T[n]_U)=\widehat{mn}([mn]_{TU})=|mn|_{TU}=|m|_T|n|_U=\hat{m}([m]_T)\hat{n}([n]_U).\]
    As $q(\mathrm{supp}^\circ(\hat{m}))=\mathcal{U}_m$ and $q(\mathrm{supp}^\circ(\hat{n}))=\mathcal{U}_n$ are bisections, this last product is a convolution, i.e. $\hat{m}([m]_T)\hat{n}([n]_U)=\hat{m}\hat{n}([m]_T[n]_U)$.  This extends to $\mathbb{T}$-multiples, showing that $\widehat{mn}$ and $\hat{m}\hat{n}$ agree on $q^{-1}(\mathcal{U}_m\mathcal{U}_n)$ and hence everywhere, as \[\mathrm{supp}^\circ(\widehat{mn})=q^{-1}(\mathcal{U}_{mn})=q^{-1}(\mathcal{U}_m\mathcal{U}_n)=q^{-1}(\mathcal{U}_m)q^{-1}(\mathcal{U}_n)=\mathrm{supp}^\circ(\hat{m})\mathrm{supp}^\circ(\hat{n}).\qedhere\]
\end{proof}

For any $C\subseteq A$, we denote the image of $C$ under the map $a\mapsto\hat{a}$ by
\[\widehat{C}=\{\hat{c}:c\in C\}.\]

\begin{lemma}\label{nhat}
    For all $n\in N$,
    \[\widehat{n^>}=\{f\in C(\Sigma;G):\mathrm{supp}^\circ(f)\Subset q^{-1}(\mathcal{U}_n)\}.\]
\end{lemma}

\begin{proof}
    If $m<n$, then $q(\mathrm{supp}^\circ(\hat{m}))=\mathcal{U}_m\Subset\mathcal{U}_n$ by \cref{prop:rainbow} and \cref{prop:supports}, thus proving the $\subseteq$ part.  Conversely, take any $f\in C(\Sigma;G)$ with $\mathrm{supp}^\circ(f)\Subset q^{-1}(\mathcal{U}_n)$.  By \cref{DominatedApproximation}, we have a sequence in $n^>$ converging to $n$ and hence the same applies to their images under the map $a\mapsto\hat{a}$, as this map is contractive (with respect to $\|\cdot\|$ on $A$ and $\|\cdot\|_\infty$ on $C(\Sigma;G)$) and hence continuous, by \cref{hatLinear}.  As $\mathrm{supp}^\circ(f)\Subset q^{-1}(\mathcal{U}_n)=\mathrm{supp}^\circ(n)$, we must therefore have some $m<n$ such that $\mathrm{supp}^\circ(f)\subseteq\mathrm{supp}^\circ(\hat{m})$.  By \cref{Interpolation}, we then have $l,s\in N$ with $m<_sl<n$.  As $sm\in B$, it follows that
    \[\mathrm{supp}^\circ(\hat{s}f)=\mathrm{supp}^\circ(\hat{s})\mathrm{supp}^\circ(f)\subseteq\mathrm{supp}^\circ(\hat{s})\mathrm{supp}^\circ(\hat{m})=\mathrm{supp}^\circ(\widehat{sm})\subseteq G^{(0)}.\]
    By \cref{Bhat}, we then have $b\in B$ with $\hat{b}=\hat{s}f$ and hence $\widehat{lb}=\widehat{ls}f$.  However, $\widehat{ls}\hat{m}=\widehat{lsm}=\hat{m}$ so $\widehat{ls}$ has to be $1$ on $\mathsf{r}(\mathrm{supp}^\circ(\hat{m}))$ and hence on $\mathsf{r}(\mathrm{supp}^\circ(f))$, which means that $\widehat{lb}=\widehat{ls}f=f$.  As $l<n$ and hence $lb<n$, we have shown $f\in\widehat{n^>}$, as required.
\end{proof}

Now we can extend this to the compatible sums defined in \cref{SigmaN}.

\begin{prop}\label{Nchat}
    The map $a\mapsto\hat{a}$ takes compatible sums from $N^>$ onto $\Nc(\Sigma;G)$, i.e.
    \[\widehat{\mathrm{csum}(N^>)}=\Nc(\Sigma;G).\]
\end{prop}

\begin{proof}
    If $n\in N^>$, then $q(\overline{\mathrm{supp}}(\hat{n}))$ is a compact bisection by \cref{nhat}.  If $m\in N^>$ is compatible with $n$ then $\overline{\mathrm{supp}}(\hat{m})^{-1}\overline{\mathrm{supp}}(\hat{n})\subseteq\overline{\mathrm{supp}}(\widehat{m^*n})=\mathcal{U}_{m^*n}\subseteq G^{(0)}$.  Likewise, $\overline{\mathrm{supp}}(\hat{m})\overline{\mathrm{supp}}(\hat{n})^{-1}\subseteq G^{(0)}$ and hence $\overline{\mathrm{supp}}(\widehat{m+n})\subseteq\overline{\mathrm{supp}}(\hat{m})\cup\overline{\mathrm{supp}}(\hat{n})$ is also a compact bisection.  Extending to finite compatible sums shows that $\widehat{\mathrm{csum}(N^>)}\subseteq \Nc(\Sigma;G)$.

    Conversely, take any $f\in \Nc(\Sigma;G)$.  As $\overline{\mathrm{supp}}(f)$ is a compact bisection, it is contained in an open bisection $O$, by \cite[Proposition 6.3]{BS2019}.  As $\{\mathcal{U}_n\}_{n\in N}$ forms a basis for the locally compact Hausdorff space $G$, a partition of unity argument yields $f_1,\ldots,f_k\in C(\Sigma;G)$ and $n_1,\ldots,n_k\in N$ such that $f=\sum_{j=1}^kf_j$ and $q(\mathrm{supp}^\circ(f_j))\Subset O\cap\mathcal{U}_{n_j}$, for all $j\leq k$.  Then \cref{nhat} yields $m_j<n_j$ with $\widehat{m_j}=f_j$, for each $j\leq k$.  Also, for all $i,j\leq k$,
    \[\mathcal{U}_{m_i^*m_j}=\mathrm{supp}^\circ\widehat{m_i^*m_j}=\mathrm{supp}^\circ(\widehat{m_i})^{-1}\mathrm{supp}^\circ(m_j)\subseteq O^{-1}O\subseteq G^{(0)}\]
    and hence $m_i^*m_j\in B$, by \cref{BG0}.  Likewise, $m_im_j^*\in B$, for all $i,j\leq k$, showing that $m_1,\ldots,m_k$ are compatible and hence $f=\sum_{j=1}^kf_j=\widehat{\sum_{j=1}^km_j}\in\widehat{\mathrm{csum}(N^>)}$.
\end{proof}

This then even extends to an algebraic isomorphism on a dense subalgebra of $A$. We restate our standing hypotheses in the following theorem, ensuring it and the subsequent corollary, which is our main result, are self-contained.

\begin{thm}\label{spanN}
    Let $A$ be a C*-algebra containing a Cartan semigroup $N$ with semi-Cartan subalgebra $B$ generated by the positive elements of $N$ and a stable expectation $E:A\twoheadrightarrow B$. Then
    the map $a\mapsto\hat{a}$ is a *-algebra isomorphism from $\mathrm{span}(N)$ onto some *-subalgebra of $\mathrm{span}(N_0(\Sigma;G))$ which takes $\mathrm{span}(N^>)$ onto $C_c(\Sigma;G)$, i.e.
    \[\widehat{\mathrm{span}(N^>)}=C_c(\Sigma;G).\]
\end{thm}

\begin{proof}
    By \cref{hatLinear} and \cref{prop:semihomo}, the map $a\mapsto\hat{a}$ is a *-algebra homomorphism on $\mathrm{span}(N)$.  To see it is also injective on $\mathrm{span}(N)$, first note that, for each $n\in N\setminus\{0\}$, there is an ultrafilter $U\in\mathcal{U}_n$, by \cref{UltrafiltersExist}.  Then $\hat{n}([n]_U)=|n|_U\neq0$ and, in particular, $\hat{n}\neq0$.  Now assume that we have shown $\hat{a}\neq0$ for all non-zero elements of $N_k:=\{\sum_{j=1}^kn_j:n_1,n_2,\ldots,n_k\in N\}$, and take any $a\in N_{k+1}\setminus\{0\}$.  Then $a^*a\neq0$ and hence we must have $n_j^*a\neq0$, for some $j\leq k+1$.  Then we can renumber if necessary to ensure that $n_{k+1}^*a\neq0$.  If $E(n_{k+1}^*a)\neq0$ then $\hat{E}(\hat{n}_{k+1}^*\hat{a})=\widehat{E(n_{k+1}^*a)}\neq0$, by the injectivity on $B\subseteq N$, and hence $\hat{a}\neq0$.  On the other hand, if $E(n_{k+1}^*a)=0$ then, as $E(n_{k+1}^*n_{k+1})=n_{k+1}^*n_{k+1}$,
    \[n_{k+1}^*a=n_{k+1}^*a-E(n_{k+1}^*a)=\sum_{j=1}^{k+1}(n_{k+1}^*n_j-E(n_{k+1}^*n_j))=\sum_{j=1}^k(n_{k+1}^*n_j-E(n_{k+1}^*n_j)).\]
    But $n_{k+1}^*n_j-E(n_{k+1}^*n_j)\in N$, for all $j\leq k$, by \cref{lem:RR} and \cref{prop:resE}, so $n_{k+1}^*a\in N_k$.  By assumption, it follows that $\hat{n}_{k+1}^*\hat{a}=\widehat{n_{k+1}^*a}\neq0$ and hence again $\hat{a}\neq0$.  This completes the induction showing $\hat{a}\neq0$, for all $a\in\mathrm{span}(N)\setminus\{0\}$, as required.

    Now just note that \cref{DominatedApproximation}, \cref{hatLinear} and \cref{nhat} yield
    \[\widehat{N}\subseteq\widehat{\mathrm{cl}(N^>)}\subseteq\mathrm{cl}_\infty(\widehat{N^>})\subseteq\mathrm{cl}_\infty(\Nc(\Sigma;G))=N_0(\Sigma;G),\]
    so $\widehat{\mathrm{span}(N)}=\mathrm{span}(\widehat{N})\subseteq\mathrm{span}(N_0(\Sigma;G))$.  Also $\mathrm{span}(N^>)=\mathrm{span}(\mathrm{csum}(N^>))$ and so \cref{Nchat} yields $\widehat{\mathrm{span}(N^>)}=\mathrm{span}(\widehat{\mathrm{csum}(N^>)})=\mathrm{span}(\Nc(\Sigma;G))=C_c(\Sigma;G)$.
\end{proof}

\begin{cor}\label{cor:it}
We have an isomorphism $\Psi$ from $A$ onto a twisted groupoid C*-algebra $C=\mathrm{cl}(C_c(\Sigma;G))$ such that $\hat{a}=j\circ\Psi(a)$, for all $a\in A$.  Moreover,
    \begin{enumerate}
        \item If $N$ is summable then $\Psi(N)$ is the monomial semigroup $\mathrm{cl}(\Nc(\Sigma;G))$.
        \item\label{it2:it} If $E$ is faithful then $C=\Psi(A)$ is the reduced C*-algebra $C_r^*(\Sigma;G)$.
    \end{enumerate}
\end{cor}

\begin{proof}
    By \cref{spanN}, $\mathrm{span}(N^>)$ is a dense copy of $C_c(\Sigma;G)$ in $A$ which is thus isomorphic to a C*-completion $C$ of $C_c(\Sigma;G)$.  Moreover, this completion is $D$-contractive by \cref{ExpectationRestriction} so $C$ is indeed a twisted groupoid C*-algebra.  Denoting the isomorphism from $A$ to $C$ by $\Psi$, we see that $\hat{a}=j\circ\Psi(a)$ for all $a\in\mathrm{span}(N^>)$ (by definition).  But as $\Psi$, $j$ and $a\mapsto\hat{a}$ are all contractive maps, $\hat{a}$ and $j\circ\Psi$ must then coincide everywhere.

    Now if $N$ is summable then
    \[\Nc(\Sigma;G)\subseteq\widehat{\mathrm{csum}(N^>)}\subseteq\widehat{N}\subseteq\widehat{\mathrm{cl}(N^>)}\subseteq\mathrm{cl}_\infty(\Nc(\Sigma;G)).\]
    It follows that $\Nc(\Sigma;G)\subseteq\Psi(N)\subseteq\mathrm{cl}(\Nc(\Sigma;G))\subseteq C$ and hence $\Psi(N)=\mathrm{cl}(\Nc(\Sigma;G))$, as $N$ is closed and $\Psi$ is an isomorphism.

    For item~\eqref{it2:it}, if $E$ is faithful then we have a Hilbert $B$-module $H$ coming from the completion of $A$ with respect to the inner product $\langle a,b\rangle:= E(a^*b)$.  Also, for each $a\in A$, we have $a_L\in \mathcal{L}(H)$, the adjointable operators on $H$, defined by $a_L(b)=ab$ for all $b\in A\subseteq H$.  The map $a\mapsto a_L$ is then an isomorphism from $A$ to a C*-subalgebra of $\mathcal{L}(H)$ and hence the norm of any $a\in A$ is equal to the operator norm of $a_L$, which is precisely $\sup\{\sqrt{\|E(c^*a^*ac)\|}:c\in\mathrm{span}(N^>) \text{ and } E(c^*c)\leq 1\}$ mentioned above.  As $\Psi$ takes $\mathrm{span}(N^>)$ to $C_c(\Sigma;G)$ and $E(a)$ to $\hat{E}(\hat{a})$, for all $a\in A$, this shows that $C$ is the completion of $C_c(\Sigma;G)$ with respect to the reduced norm.
\end{proof}

\section{Normalisers vs MASAs}\label{sec:masa}

The previous section achieves our main goal of identifying the C*-algebra $A$ with a twisted groupoid C*-algebra.  In addition to extending the Kumjian-Renault theory to non-reduced C*-algebras, it extends the theory in two other important ways.  Firstly, our semi-Cartan subalgebra $B$ is only required to be commutative, not a MASA like in the definition of a Cartan subalgebra.  On the other hand, in the Kumjian-Renault theory, one works with the entire normaliser semigroup $N(B)$, while our Cartan semigroup $N$ need only be contained in $N(B)$, as noted in \cref{lem:Normal}.  Here we show that these latter two aspects of our generalisation are actually one and the same, modulo summability.

First observe that if $B$ is a MASA, then \cref{SummableNormalisers} tells us that the normalisers $N(B)$ form summable Cartan semigroup, which must therefore contain not only $N$ but also $\overline{\mathrm{csum}}(N)$, the closure of its compatible sums.  Using \cref{prop:Shiftable}, we can show that, in fact, $N(B)$ contains no other elements.  In particular, $N(B)$ recovers the original $N$ if $N$ is summable and $B$ is a MASA.

\begin{prop}\label{MASAimpliesNormalisers}
    If $B$ is a MASA then $\overline{\mathrm{csum}}(N)=N(B)$.
\end{prop}

\begin{proof}
    Assume $\overline{\mathrm{csum}}(N)\neq N(B)$, so we have some $a\in N(B)\setminus\overline{\mathrm{csum}}(N)$. Let $p_k, q_k$ be as in \cref{def:p_k poly}. Noting that $aa^*a\in\overline{\mathrm{csum}}(N)$ would imply
    \[a=\lim_kap_k(a^*a)=\lim_kaa^*aq_k(a^*a)\in\mathrm{cl}(\overline{\mathrm{csum}}(N)B)\subseteq\overline{\mathrm{csum}}(N),\]
    it follows that $aa^*a\notin\overline{\mathrm{csum}}(N)$ and hence $E(a^*a)a^*=a^*aa^*\notin\overline{\mathrm{csum}}(N)^*=\overline{\mathrm{csum}}(N)$.
    
    We claim that this means $E(na)a^*\notin N$, for some $n\in N$.  If not then, for any $n_1,\ldots,n_k\in N$, we would have $E(\sum_{j=1}^kn_ka)a^*=\sum_{j=1}^kE(n_ka)a^*\in\mathrm{csum}(N)$ because, for all $i,j\leq k$, $aE(n_ia)^*E(n_ja)a^*\in aBa^*\subseteq B$ and $E(n_ia)a^*aE(n_ja)\in BBB\subseteq B$.  But we know that $a^*\in A=\mathrm{cl}(\mathrm{span}(N))$ and so this would mean $E(a^*a)a^*\in\overline{\mathrm{csum}}(N)$, a contradiction.  This completes the proof of the claim.
    
    Now take $n$ with $E(na)a^*\notin N$ and note $E(na)a^*=\lim_kE(p_k(nn^*)na)a^*$ so
    \[N\not\ni E(nn^*na)a^*=nn^*E(na)a^*=E(na)nn^*a^*=nE(an)n^*a^*.\]
    Thus $E(an)n^*a^*\notin N\supseteq B$, even though $E(an)n^*a^*$ commutes with every $b\in B$, as in the proof of \cref{prop:CartanGeneralisation} (replace $n$ with $an$ in \eqref{Commutant}), showing $B$ is not a MASA.
\end{proof}

As alluded to above, the converse also holds, i.e. if $N=N(B)$ then $B$ is a MASA.  To prove this it will help to first prove the following preparatory lemma.

We denote the commutant of $B$
\[C(B)=\{a\in A:\forall b\in B\ (ab=ba)\}.\]
Note that
\[C(B)\cap N(B)=\{c\in C(B):c^*c,cc^*\in B\}.\]

\begin{lemma}\label{lem:12.12}
    If $N=N(B)$ but $B$ is not a MASA then we have $K\in\mathbb{N}\cup\{\infty\}$ with $K>1$ and $c\in C(B)\setminus\{0\}$ such that $c^*c=cc^*\in B$ and $E(c^k)=0$, for all $k<K$.  If $K\neq\infty$ then we can further ensure that $c^K\in B_+$.
\end{lemma}

\begin{proof}
    If $B$ is not a MASA in $A$ then the same applies to their unitisations, i.e. $\widetilde{B}$ is not a MASA in $\widetilde{A}$.  Thus we have a unitary $u\in C(\widetilde{B})\setminus\widetilde{B}$, necessarily of the form $u=t1+a$, for some $t\in\mathbb{T}$ and $a\in C(B)\setminus B$.  Now note that there must be some $b\in B_+$ with $bu\notin B$ (otherwise, taking an approximate unit $(b_\lambda)\subseteq B_+$ for $A$, we would have $a=\lim_\lambda(b_\lambda u-tb_\lambda)\in B$, a contradiction).  Setting $c=bu-E(bu)$, it follows that $E(c)=0$, $c\in C(B)\setminus\{0\}$,
    \[cc^*=b^2-buE(u^*b)-E(bu)u^*b+E(bu)E(u^*b)=b^2-E(bu)E(u^*b)\in B\]
    and, likewise, $c^*c=b^2-E(u^*b)E(bu)=cc^*$.
    
    If $E(c^k)=0$ for all $k>1$ as well, we are done.  If not then we can take minimal $K\in\mathbb{N}$ such that $E(c^K)\neq0$.  As $c\in C(B)$ and $c^*c=cc^*\in B$ it follows that $c\in N(B)=N$ and hence $E(c^K)\sqsubseteq c^K$.  Take $(b_j)\in B_+^1$ with $E(c^K)=\lim_jb_jc^K$.  Noting that the C*-algebra generated by $B$ and $c$ is commutative, we see that the sequence $(b_jc)$ must also have a limit.  Letting $d=\lim_jb_jc\in C(B)$, we again see that $d^*d=dd^*\in B$ and $E(d^k)=0$, for all $k<K$.  Moreover, $d^K=\lim_jb_j^Kc^K=E(c^K)\in B\setminus\{0\}$, which also implies $d\neq0$.  In other words, replacing $c$ with $d$, we may further assume that $c^K\in B$.  Multiplying $c$ by some element of $B$ if necessary, we may further assume that the spectrum of $c^K$ lies entirely within the left or right half of the complex plane.  This ensures that there is some $b\in B$ such that $b^K=c^K$ and hence $(b^*c)^K=c^{K*}c^K\in B_+\setminus\{0\}$.  So replacing $c$ with $b^*c$ if necessary, we also obtain $c^K\in B_+$, as required.
\end{proof}

\begin{thm}\label{NormalisersImpliesMASA}
    If $N=N(B)$ then $B$ is a MASA.
\end{thm}

\begin{proof}
    Looking for a contradiction, assume that $N=N(B)$ but $B$ is not a MASA, take $K>1$ and $c\in C(B)\setminus\{0\}$ as in \cref{lem:12.12}.
First assume that $K=2$.  Note $|c|=\sqrt{c^2}\in B_+$, as $c^2\in B_+$ and $cc^*=c^*c$.  Letting $n=|c|+ic\in C(B)$, we see that $n^*n=|c|^2+i|c|c-ic^*|c|+c^*c=2|c|^2\in B_+$ and, likewise, $nn^*=2|c|^2\in B_+$ so $n\in N(B)=N$ and hence $|c|=E(n)\sqsubseteq n$.  Taking $(b_k)\subseteq B^1_+$ with $E(n)=\lim_kb_kE(n)=\lim_kb_kn$, we see that $(b_k)$ is an approximate unit for $|c|=E(n)$ and hence $c$ and thus $n$ as well.  So $|c|=E(n)=\lim_kb_kn=n=|c|+ic$ and hence $ic=0$, contradicting our choice of $c\neq0$.
    
    Now assume that $2<K<\infty$.  Again note $c^K=|c|^K$, as $c^K\in B_+$, and hence $|c|^kc^K=|c|^{K+k}=|c|^{K-k}c^{*k}c^k$, whenever $1\leq k\leq K$.  Viewing $c$ as an element of $C_0(X_{C^*(c)})$, we can cancel $c^k$ on either side to obtain $|c|^kc^{K-k}=|c|^{K-k}c^{*k}$ and hence
    \[\Big(\sum_{k=1}^{K-1}|c|^{K-k}c^k\Big)^*=\sum_{k=1}^{K-1}|c|^{K-k}c^{*k}=\sum_{k=1}^{K-1}|c|^kc^{K-k}=\sum_{k=1}^{K-1}|c|^{K-k}c^k.\]
    It follows that $n=n^*$ when we define
    \[n=(K-2)|c|^K-2\sum_{k=1}^{K-1}|c|^{K-k}c^k\in C(B).\]
    To help compute $n^*n$, notice that whenever $j$ and $k$ are numbers from $1$ to $K-1$, we see that $|c|^{j+k}c^{2K-j-k}=|c|^{2K-l}c^l$, for some $l$ between $0$ and $K-1$, e.g.~$l=0$ if $j+k=K$, $l=1$ if $j+k+1=K$, $l=2$ if $j+k+2=K$ or $2K$, etc.  It follows that
    \[\sum_{j,k=1}^{K-1}|c|^{j+k}c^{2K-j-k}=(K-1)|c|^{2K}+(K-2)\sum_{k=1}^{K-1}|c|^{2K-k}c^k.\]
    Then we see that
    \begin{align*}
        n^2&=(K-2)^2|c|^{2K}-4(K-2)\sum_{k=1}^{K-1}|c|^{2K-k}c^k+4((K-1)|c|^{2K}+(K-2)\sum_{k=1}^{K-1}|c|^{2K-k}c^k))\\
        &=((K-2)^2+4(K-1))|c|^{2K}\\
        &=K^2|c|^{2K}.
    \end{align*}
    In particular, $n^*n=nn^*=n^2\in B_+$ so $n\in N(B)\subseteq N$.  As $E(c^k)=0$, for all $k<K$, it follows that $(K-2)|c|^K=E(n)\sqsubseteq n$.  Taking $(b_k)\subseteq B^1_+$ with $E(n)=\lim_kb_kE(n)=\lim_kb_kn$, we again see that $b_k$ is an approximate unit for $|c|$ and hence $c^k$, for all $k<K$, and thus for $n$ as well.  So
    \[(K-2)|c|^K=E(n)=\lim_kb_kn=n=(K-2)|c|^K-2\sum_{k=1}^{K-1}|c|^{K-k}c^k\]
    and hence $\sum_{k=1}^{K-1}|c|^{K-k}c^k=0$.  But again this is not possible because
    \[E(c^*\sum_{k=1}^{K-1}|c|^{K-k}c^k)=E(\sum_{k=1}^{K-1}|c|^{K-k+1}c^{k-1})=|c|^K\neq0.\]

    The only remaining possibility is $K=\infty$.  In this case let us first replace $c$ with $\varepsilon c$ for some $\varepsilon>0$ small enough to ensure that $\|c\|<\frac{1}{\sqrt{2}}$.  Then let $n=|c|e^{i(c+c^*)}\in C(B)$ so $n^*n=nn^*=|c|^2\in B$ and again $n\in N(B)\subseteq N$.  Note that we have $(a_k)_{k\in\mathbb{Z}_+}\subseteq B$ with
    \[n=a_0+\sum_{k=1}^\infty a_k(c^k+c^{*k}),\]
    e.g. $a_0=\sum_{k=0}^\infty(-1)^kk!^{-2}|c|^{2k+1}$ and $a_1=i\sum_{k=0}^\infty(-1)^kk!^{-1}(k+1)!^{-1}|c|^{2k+1}$.  Now note that, whenever $0<r<\frac{1}{\sqrt{2}}$,
    \begin{align*}
	\Big|\sum_{k=0}^\infty(-1)^kk!^{-2}r^{2k+1}\Big|&\geq r-\Big|\sum_{k=1}^\infty(-1)^kk!^{-2}r^{2k+1}\Big|\geq r-\sum_{k=1}^\infty r^{2k+1}\\
	&=r-r^3/(1-r^2)=(r-2r^3)/(1-r^2)>r-2r^3>0.
\end{align*}
It follows that $a_0$ has the same support as $|c|$ when identifying the C*-algebra $D$ generated by $B$ and $c$ with $C_0(X_D)$.  Exactly the same argument shows that $a_1$ has the same support as $|c|$ and hence $a_0$ as well.  This means that any approximate unit in $B$ for $a_0$ is also an approximate unit for $a_1$.  As $E(n)\sqsubseteq n$, we again have $(b_k)\subseteq B^1_+$ with
    \[a_0=E(n)=\lim_kb_kE(n)=\lim_kb_kn=a_0+a_1(c+c^*)+\lim_kb_k\sum_{k=2}^\infty a_k(c^k+c^{*k}).\]
    This implies $a=0$ where $a=a_1(c+c^*)+\lim_kb_k\sum_{k=2}^\infty a_k(c^k+c^{*k})$, which again is not possible because $E(ca)=a_1|c|^2\neq0$.  Thus we get a contradiction for all possible $K$.
\end{proof}

Recall that \cite[Theorem 3.1 (2)$\Leftrightarrow$(5)]{ABCCLMR2023} says that the C*-algebraic
local bisection hypothesis holds for a reduced twisted groupoid C*-algebra precisely when the canonical diagonal is a MASA and hence a Cartan subalgebra.  In light of \cref{TwistedGroupoid->CartanSemigroup}, the results above extend this to more general (e.g. full) C*-completions of $C_c(\Sigma;G)$.  Even in the reduced case (which corresponds to $E$ being faithful, by \cref{cor:it}), the above results provide a somewhat different proof of \cite[Theorem 3.1 (2)$\Leftrightarrow$(5)]{ABCCLMR2023}.

We also get the following corollary.

\begin{cor}\label{cor:masa}
    Let $A$ be a C*-algebra with commutative C*-subalgebra $B$ such that the span of $N(B)$ is dense in $A$.  Suppose we are also given an expectation $E:A\twoheadrightarrow B$.  Then $B$ is a Cartan subalgebra if and only if $E$ is faithful and $E(n)\sqsubseteq n$, for all $n\in N(B)$.
\end{cor}

\begin{proof}
    If $B$ is a Cartan subalgebra then, by definition, there is a faithful expectation onto $B$.  As shown in \cref{prop:CartanGeneralisation}, both this faithful expectation and the given expectation $E$ are then stable and hence $N(B)$ is a Cartan semigroup.  By \cref{prop:resE}, these expectations must then be the same and satisfy $E(n)\sqsubseteq n$, for all $n\in N(B)$.

    Conversely, for all $n\in N(B)$, if $E(n)\sqsubseteq n$ then $E(n)n^*\in\mathrm{cl}(nBn^*)\subseteq B$.  Also $n^*n$ commutes with $B$ by \cite[Proposition 2.1]{Pitts}, so $E(n^*n)\sqsubseteq n^*n$ implies $E(n^*n)\leq n^*n$ and hence $a:=n^*n-E(n^*n)\geq0$.  As $E(a)=0$, if $E$ is faithful then $a=0$ and hence $n^*n=E(n^*n)\in B$.  Thus $N(B)$ is a Cartan semigroup with associated semi-Cartan subalgebra $B$.  By \cref{NormalisersImpliesMASA}, $B$ is then a MASA and hence a Cartan subalgebra.
\end{proof}

\section{Domination Variants}

We finish by examining a few further properties of domination which clarify its connections to similar relations considered previously in the literature.

The first attempt to define Kumjian-Renault's Weyl groupoid from a domination-like relation appeared in \cite{Bice2021}.  This variant of domination, which we denote here by $<_*$, was a stronger relation defined only on the unit ball of $N$ by
\[m<_*n\qquad\Leftrightarrow\qquad m<_{n^*}n.\]
This is closer to the relation $\ll$, where $m\ll n$ means $mn=m$, which was originally considered on semigroups of real-valued continuous functions in \cite{Milgram1949} (and also sometimes used in C*-algebras -- see \cite[II.3.4.3]{Blackadar2017}).  However, more recent work in \cite{Bice2023}, \cite{Bice2022}, \cite{BC2021} and the present paper indicates that $<$ is the better relation to work with.  Nevertheless, $<$ and $<_*$ are still closely related and, on Cartan semigroups at least, they really differ only by a factor of $B_+$, as we now show.  We also simultaneously show that we could have required $sn$ and $ns$ to lie in $B^1_+$ when defining $<$ (much as the sequence defining $\sqsubseteq$ can be taken in $B^1_+$, as shown in \cref{PositiveBallSequence}).  Accordingly, let us define a strengthening $<^1_s$ of $<_s$ by
\[m<_s^1n\qquad\Leftrightarrow\qquad ms,sm\in B,\quad sn,ns\in B^1_+\quad\text{and}\quad nsm=m=msn.\]
Below we also denote infima and suprema by $\wedge$ and $\vee$ respectively.

\begin{thm}\label{BallDomination}
If $m<n$ then we have $s\in n^*B_+\cap B_+n^*$ with $m<_s^1n$.
\end{thm}

\begin{proof}
First we claim that
\[m<_sn\qquad\Rightarrow\qquad m<_{ss^*n^*}n.\]
Indeed, if $m<_sn$ then $nss^*n^*\in B_+$ and $ss^*n^*n\in B_+B_+\subseteq B_+$.  Also, \cref{cor:binormal} yields $ss^*n^*m\in sBm\subseteq B$ and \cref{mnstar} yields $mss^*n^*=s^*n^*ms=ms\in B$ and
\[nss^*n^*m=nsm=m=s^*n^*msn=mss^*n^*n,\]
i.e. $m<_{ss^*n^*}n$.  This shows that $m<n$ implies $m<_sn$, for some $s\in N$ with $ns,sn\in B_+$.

Defining $f$ on $\mathbb{R}_+$ by $f(x)=x\wedge x^{-1}$ (in particular $f(0)=0\wedge\infty=0$), we next claim
\[m<_sn\quad\text{and}\quad ns,sn\in B_+\qquad\Rightarrow\qquad m<_{sf(ns)}n.\]
Indeed, if $m<_sn$ and $ns,sn\in B_+$ then $msf(ns),nsf(ns)\in BB\subseteq B$ and \cref{cor:binormal} yields $sf(ns)n,sf(ns)m\in B$.  Moreover, as $f(1)=1$, $bm=m$ implies $f(b)bm=m$ and $mb=m$ implies $mbf(b)=m$, for any $b\in B$, so $nsf(ns)m=m=mf(sn)sn=msf(ns)n$, i.e. $m<_{sf(ns)}n$.  As $f(x)\leq x^{-1}$, for all $x>0$, $\|nsf(ns)\|\leq1$ and $\|sf(ns)n\|=\|f(sn)sn\|\leq1$.  This shows that $m<n$ implies $m<_s^1n$, for some $s\in N$.

Now assume $m<_s^1n$.  Take any $r>16\|s\|^4$ and define functions $g$ and $h$ on $\mathbb{R}_+$ by
\[g(x)=(2x-1)\vee0\qquad\text{and}\qquad h(x)=x^{-1}\wedge rx.\]
Noting that $g(sn)n^*nn^*=n^*ng(sn)n^*=n^*g(ns)nn^*=n^*nn^*g(ns)$, it follows that $g(sn)p(n^*n)n^*=n^*p(nn^*)g(ns)$ for any polynomial $p$ without constant term and hence any continuous $p$ with $p(0)=0$.  To complete the proof it suffices to show $m<_t^1n$ where
\[t=g(sn)h(n^*n)n^*=n^*h(nn^*)g(ns)\in n^*B_+\cap B_+n^*.\]

To see this, first note $(sn)^2=n^*s^*sn\leq\|s\|^2n^*n$.  Identifying $B$ with $C_0(X_B)$, we see that if $x\in X$ satisfies $sn(x)\geq1/2$ then $1/4\leq sn(x)^2\leq\|s\|^2n^*n(x)$ so $1/16\leq\|s\|^4n^*n(x)^2$ and hence $n^*n(x)^{-1}\leq16\|s\|^4n^*n(x)\leq rn^*n(x)$, which implies $h(n^*n)n^*n(x)=1$.  On the other hand, if $sn(x)\leq1/2$ then $2sn(x)-1\leq0$ and hence $g(sn)(x)=0$.  These observations together imply that $tn=g(sn)h(n^*n)n^*n=g(sn)\in B^1_+$ and, likewise, $nt=g(ns)\in B^1_+$.  As $g(1)=1$ and $nsm=m=msn$, it follows that $ntm=g(ns)m=m=mg(sn)=mtn$. 
 Finally note $n^*m=n^*nsm\in BB\subseteq B$ and hence $tm=g(sn)h(n^*n)n^*m\in B$.  Likewise $mt\in B$ and hence $m<_t^1n$, as required.
\end{proof}

We can even extend this result to any bounded finite family.

\begin{cor}
    If $m_1,\ldots,m_k<n$, we have $s\in n^*B_+\cap B_+n^*$ with $m_1,\ldots,m_k<_s^1n$.
\end{cor}

\begin{proof}
    If $m_1,\ldots,m_k<n$ then we have $s_1,\ldots,s_k\in n^*B_+\cap B_+n^*$ with $m_j<_{s_j}^1n$, for all $j\leq k$.  This means we have $b_1,\ldots,b_k,c_1,\ldots,c_k\in B_+$ with $s_j=b_jn^*=n^*c_j$, for all $j\leq k$.  Letting $s=(b_1\vee\ldots\vee b_k)n^*=n^*(c_1\vee\ldots\vee c_k)\in n^*B_+\cap B_+n^*$, we see that, for all $j\leq k$,
    \[m_j=m_js_jn=m_jb_jn^*n=m_j(b_1n^*n\vee\ldots\vee b_kn^*n)=m_j(b_1\vee\ldots\vee b_k)n^*n=m_jsn,\]
    as $b_1n^*n,\ldots,b_kn^*n\in B^1_+$.  Likewise, $nsm_j=m_j$ and hence $m_j<_s^1n$, for all $j\leq k$.
\end{proof}

One immediate application of this is that sums behave well with respect to domination.

\begin{cor}
    For all $l,m,n\in N$,
    \[l,m<n\qquad\Rightarrow\qquad l+m<n.\]
\end{cor}

\begin{proof}
    If $l,m<n$ then, by the above result, we have $s\in N$ with $l,m<_sn$ and hence $l+m=lsn+msn=(ls+ms)n\in(B+B)N\subseteq BN\subseteq N$.  Likewise $l+m=ns(l+m)$ etc. and hence $l+m<_sn$.
\end{proof}

As another application, we can show that $(N,<)$ is a predomain in the sense of \cite{Keimel2016}, also known as an abstract basis in \cite[Definition III-4.15]{GHKLMS2003} and \cite[Lemma 5.1.32]{Goubault2013}.

\begin{cor}\label{Predomain}
If $m_1,\ldots,m_k<n$, we have $l\in nB_+\cap B_+n$ with $m_1,\ldots,m_k<_{l^*}l<n$.
\end{cor}

\begin{proof}
If $m_1,\ldots,m_k<n$ then we have $s\in n^*B_+\cap B_+n^*$ with $m_1,\ldots,m_k<_s^1n$, by \cref{BallDomination}.  Then we have $q\in nB_+\cap B_+n$ with $m_1,\ldots,m_k<_s^1q<n$, as in the proof of \cref{Interpolation}.  This means we have $b,c,d,e\in B$ such that $s=bn^*=n^*c$ and $t=dn=ne$.  Setting $l=n\sqrt{be}=\sqrt{cd}n$, it follows that $m_1,\ldots,m_k<_{l^*}l<n$.
\end{proof}

It follows that we obtain a domain from the $<$-ideals of $N$, which correspond to the open bisections of our ultrafilter groupoid.


\vspace{2ex}
\bibliographystyle{plainurl}
\makeatletter\renewcommand\@biblabel[1]{[#1]}\makeatother
\bibliography{references}

\end{document}